\documentclass[8pt]{article}

\usepackage[utf8]{inputenc}
\usepackage[T2A]{fontenc}
\usepackage[margin=1.2cm]{geometry}
\usepackage{mathtools}
\usepackage{authblk}

\usepackage{subfigure}

\usepackage{pgf,tikz,pgfplots}
\usetikzlibrary{arrows.meta}
\usetikzlibrary{calc}
\pgfplotsset{compat=1.15}

\usetikzlibrary{arrows,patterns}
\usepackage{float}

\usepackage{amsmath,amsthm,amssymb}
\usepackage{hyperref}

\newtheorem{theorem}{Theorem}[section]
\newtheorem{definition}[theorem]{Definition}

\newtheorem{lemma}[theorem]{Lemma}
\newtheorem{proposition}[theorem]{Proposition}
\newtheorem{corollary}[theorem]{Corollary}

\newtheorem{problem}[theorem]{Problem}

\newtheorem{question}[theorem]{Question}

\theoremstyle{remark}
\newtheorem{remark}[theorem]{Remark}
\newtheorem{example}[theorem]{Example}

\DeclareMathOperator{\Int}{Int}
\DeclareMathOperator{\conv}{conv}
\DeclareMathOperator{\Per}{Per}
\newcommand\N{\mathbb{N}}
\newcommand\R{\mathbb{R}}

\newcommand{\F}{{\mathcal F}}
\newcommand{\St}{{\mathcal St}}

        \newcommand{\HH}{{\mathcal H}}
        \renewcommand{\H}{\HH^1}
        
        \newcommand{\defeq}{:=}
        \newcommand{\forget}[1]{}

        \def\dist{\mathrm{dist}\,}
        
        \def\diam{\mathrm{diam}\,}

        \def\X{\mathrm{X}}
        \def\E{\mathrm{E}}
        \def\S{\mathcal S}

        \def\diam{\mathrm{diam}\,}

        \def\turn{\mathrm{turn}\,}

        \def\Int{\mathrm{Int}\,}

        \def\oord{\mathrm{ord}\,}

\author[1]{Danila Cherkashin}
\author[2,3]{Yana Teplitskaya}

\affil[1]{Institute of Mathematics and Informatics, Bulgarian Academy of Sciences, Sofia, Bulgaria}
\affil[2]{Mathematical Institute, Leiden University, the Netherlands}
\affil[3]{Laboratoire  de Math{\'e}matiques d’Orsay, Universit{\'e} Paris-Saclay, CNRS, Orsay, France}

\title{An overview of maximal distance minimizers problem}
\begin{document}
\maketitle

\begin{abstract}
    Consider a compact $M \subset \mathbb{R}^d$ and $l > 0$. A maximal distance minimizer problem is to find a connected compact set $\Sigma$ of the length (one-dimensional Hausdorff measure $\H$) at most $l$ that minimizes
\[
\max_{y \in M} \dist (y, \Sigma),
\]
where $\dist$ stands for the Euclidean distance.

We give a survey on the results on the maximal distance minimizers and related problems. Also we fill some natural gaps by 
showing NP-hardness of the maximal distance minimizing problem, establishing its $\Gamma$-convergence, considering the penalized form and discussing uniqueness of a solution. We finish with open questions.

\end{abstract}

\section{Introduction}

This work is devoted to solutions to the following maximal distance minimizer problem. 

\begin{problem}
For a given compact set $M \subset \R^d$ and $l > 0$ to find a connected compact set $\Sigma$ of length (one-dimensional Hausdorff measure $\H$) at most $l$ that minimizes
\[
\max_{y \in M} \dist (y, \Sigma),
\]
where $\dist$ stands for the Euclidean distance.
\label{ProblemMain}
\end{problem}

Problem~\ref{ProblemMain} with the average distance minimization problem (which is discussed later) were introduced in a very general form by Buttazzo, Oudet and Stepanov in~\cite{buttazzo2002optimal}.
They were motivated by urban transportation planning.
The aim of an urban transportation network (for instance the set of metro or street traffic routes) is to provide the easiest access of the people to their destinations, the densities of which are supposed to be given finite Borel measures. 
Such a transportation network will be modeled by a one-dimensional closed connected set. It is reasonable to suppose that getting to a destination without using urban traffic in general would cost a single citizen an effort (estimated in time required and, at the end, also in terms of respective financial loss) proportional to the actual walking distance while the cost of using the urban traffic (the price of the ticket) is independent of the actual length of the traffic route used, and, moreover, for simplicity can be assumed to be negligible with respect to the cost of a transportation without using the urban traffic.

Later Problem~\ref{ProblemMain} was stated in the current form by Miranda, Paolini and Stepanov in~\cite{miranda2006one,paolini2004qualitative}. 
Of course, we may encounter this problem (and its dual form) in different life situations. For example, you are the mayor of a city completely occupied by rats, and you need to prepare the site for the most important international sporting event. Then you want to lay cables of a minimum length that will emit an ultrasonic rodent repeller in all areas of the city where athletes and tourists will be.

A \textit{maximal distance minimizer} is a solution to Problem~\ref{ProblemMain}.


\subsection{Class of problems} 

Maximal distance minimization problem could be considered as a particular example of shape optimization problems. 
A shape optimization problem is a minimization problem where the unknown variable runs over a class of domains; then every shape optimization problem can be written in the form $\min F(\Sigma) :  \Sigma \in A$ where $A$ is the class of admissible domains and $F()$ is the cost function that one has to minimize over $A$. 

So for a given compact set $M$ and positive number $l\geq 0$ let the admissible set $A$ be a set of all closed connected set $\Sigma'$ with length constraint $\H(\Sigma') \leq l$; and let cost function be the \textit{energy} $F_M(\Sigma)=\max_{y \in M} \dist(y,\Sigma)$. Also $F_M(\emptyset) := \infty$.

\subsection{Dual problem}

Define the dual problem to Problem~\ref{ProblemMain} as follows.

\begin{problem}
For a given compact set $M \subset \R^d$ and $r > 0$ to find a connected compact set $\Sigma$ of the minimal length (one-dimensional Hausdorff measure $\H$) such that
\[
\max_{y \in M} \dist (y, \Sigma) \leq r.
\]
\label{TheDualProblem}
\end{problem}

The dual problem also admits a natural interpretation. Namely, suppose that we have to provide a gas supply pipeline to every house located in some area $M$ under the condition that the gas supply should reach each house at distance not greater than a given $r > 0$. The company constructing the pipeline will naturally try to minimize its length under the above restriction, which reduces to solving the Problem~\ref{TheDualProblem}.

In a nondegenerate case (i.e. for $\H(\Sigma)>0$) the primal and dual problems have the same sets of solutions for the corresponding $r$ and $l$ (see~\cite{paolini2004qualitative}) and hence an equality $F_M(\Sigma) = r$ is reached for a solution $\Sigma$ to Problem~\ref{TheDualProblem}.

\subsection{The first parallels with the average distance minimization problem}
Maximal distance minimization problem is somehow similar to another shape optimization problem: average distance minimization problem 
(see the survey of Lemenant~\cite{lemenant2011presentationENG}) and it seems interesting to compare the known results and open questions concerning these two problems. 
In the average distance minimization problem's statement the admissible set $A$ is the same as in maximal distance minimization problem, but the cost function is defined as
\[
F(\Sigma_a) = \int_{M} A(\dist(y,\Sigma_a)) d\phi(x), 
\]
where $A: \mathbb{R}^+\rightarrow \mathbb{R}^+$ is a nondecreasing function and $\phi()$ is  a finite nonnegative measure with compact nonempty support in $\mathbb{R}^d$.

Minimization problems for average distance and maximum distance functionals are used in economics and urban planning with similar interpretations. If it is required to find minimizers under the cardinality constraint $\sharp \Sigma \leq k$, instead of the length and the connectedness constraints, where $k \in \mathbb{N}$ is given and $\sharp$ denotes the cardinality, then the corresponding problems are referred to as optimal facility location problems.

\subsection{Notation}

For a given set $X \subset \mathbb{R}^d$ we denote by $\overline{X}$ its closure, by $\Int(X)$ its interior and by $\partial X$ its topological boundary.

Let $B_\rho(x)$ stand for the open ball of radius $\rho$ centered at a point $x$, and let $B_\rho(T)$
be the open $\rho$-neighborhood of a set $T$ i.e.
\[
B_\rho(T) := \bigcup_{x\in T} B_\rho(x)
\]
(in other words $B_\rho(T)$ is the Minkowski sum of a ball $B_\rho$ centered in the origin and $T$). Note that the condition
\[
\max_{y \in M} \dist (y, \Sigma) \leq r
\]
is equivalent to $M \subset \overline{B_r(\Sigma)}$.

For given points $b$, $c \in \mathbb{R}^d$ we use the notation $[bc]$, $[bc)$ and $(bc)$ for the corresponding closed line
segment, ray and line respectively. 

A \textit{regular tripod} is a union of three segments $[ax] \cup [bx] \cup [cx]$ with $\angle axb = \angle axc = \angle bxc = 2\pi/3$.
Note that a regular tripod is always coplanar.

\subsection{Existence. Absence of loops. Ahlfors regularity and other simple properties}

For the both problems the existence of solutions can be obtained in a straightforward way: according to the classical Blaschke and Go{\l}{\k{a}}b theorems, the class of admissible sets is compact for the Hausdorff distance and both of the functions (maximal distance and also the average distance) is continuous for this convergence because of the uniform convergence of $x \rightarrow \dist(x, \Sigma)$.

\begin{definition}
A closed set $\Sigma$ is said to be Ahlfors regular if there exists some constants $C_1$, $C_2>0$ and a
radius $\varepsilon_0>0$ such that $C_1\varepsilon \leq \H(\Sigma \cap B_\varepsilon(x))\leq C_2\varepsilon$ for every $x\in \Sigma$ and $\varepsilon<\varepsilon_0$.
\end{definition}

In the work~\cite{paolini2004qualitative} Paolini and Stepanov proved
\begin{itemize}
    \item  the absence of closed loops for maximum distance minimizers and, under general
conditions on $\phi$, the absence of closed loops for average distance minimizers;
    \item  the Ahlfors regularity of maximum distance minimizers (in the plane the constants are absolute) and, under the additional
summability condition on $\phi$, the Ahlfors regularity of average distance minimizers. Gordeev and Teplitskaya~\cite{gordeev2022regularity} refine Ahlfors constants of planar maximum distance minimizers to the best possible, i.e. show that $\H (\Sigma \cap B_\varepsilon(x)) = \oord_x\Sigma \cdot \varepsilon + o(\varepsilon)$,
where $\oord_x\Sigma \in \{1, 2, 3\}$. Also, Gordeev and Teplitskaya proved a ``local'' analogue in $\mathbb{R}^d$, namely the same result with $\varepsilon$ depending on $x$.

\item Recall that maximal distance minimization problem and the dual problem have the same sets of solutions (the planar case was proved before by Miranda, Paolini, Stepanov in~\cite{miranda2006one}). It particularly implies that maximal distance minimizers must have maximum available length $l$. Paolini and Stepanov also proved that average distance minimizers (with additional assumptions on $\phi$) have maximum available length.
\end{itemize}

In the work~\cite{inverse2022} the following basic results were shown.

\begin{itemize}
    \item[(i)] Let $\Sigma$ be an $r$-minimizer for some $M$. Then $\Sigma$ is an $r$-minimizer for $\overline{B_r(\Sigma)}$.
    \item[(ii)] Let $\Sigma$ be an $r$-minimizer for $\overline{B_r(\Sigma)}$. Then $\Sigma$ is an $r'$-minimizer for $\overline{B_{r'}(\Sigma)}$, where $0 < r' < r$.
\end{itemize}

\subsection{Complexity of the problem}

The corresponding discretization of Problem~\ref{ProblemMain} is NP-complete (see Problem~\ref{DMDM} and Section~\ref{sec:NP}), this means that Problem~\ref{ProblemMain} itself is NP-hard.

Even in the plane and for a general 3-point set $M = \{a,b,c\}$ a solution to Problem~\ref{ProblemMain} cannot be constructed by a straightedge (ruler) and a compass.
The reason is that for a large enough $\angle abc$ and a proper $r$ Problem~\ref{ProblemMain} coincides with Alhazen's billiard  problem (a reflection from a spherical mirror), for which the unconstructability was shown by Neumann~\cite{neumann1998reflections}. This differs the situation from the Steiner tree problem (Problem~\ref{Problem1}), see~\cite{melzak1961problem}.

\subsection{Local maximal distance minimizers}

\begin{definition} \label{def:local}
Let $M \subset \mathbb{R}^d$ be a compact set and let $r > 0$.
A closed connected set $\Sigma \subset \mathbb{R}^d$ with $\H(\Sigma)<\infty$ is called a local minimizer if $\F_M(\Sigma) \leq r$
and there exists $\varepsilon > 0$ such that for any connected set $\Sigma'$ satisfying $\F_M(\Sigma') \leq r$
and $\diam (\Sigma \triangle \Sigma') \leq \varepsilon$ the inequality $\H(\Sigma) \leq \H(\Sigma')$ holds, where $\triangle$ is the symmetric difference.
\end{definition}

Any maximal distance minimizer is also a local minimizer. Usually the regularity properties of maximal distance minimizers are also true for local maximal distance minimizers (see~\cite{gordeev2022regularity}).

 \section{Regularity}

\subsection{Tangent rays. Blow up limits in \texorpdfstring{$\R^d$}{Rd}}
\begin{definition}
We say that a ray $ (ax] $ is a \emph{tangent ray} of a set $ \Gamma \subset \mathbb{R}^d$
at the point $ x\in \Gamma $ if there exists a
sequence of points $ x_k \in \Gamma \setminus \{x\}$ such that $ x_k \rightarrow x $ and $ \angle x_kxa \rightarrow 0 $.
\end{definition}

Then we have the following regularity theorem.

\begin{theorem}[Gordeev--Teplitskaya~\cite{gordeev2022regularity}]    \label{th:GT}
Let $\Sigma$ be a minimizer for a compact set $M \subset \mathbb{R}^d$ and $r > 0$. Then there are at most three
tangent rays at any point of $\Sigma$, and the pairwise angles between the tangent rays are at least $2\pi/3$. Furthermore, tangent rays coincide with one-sided tangents, particularly the angles between one-sided tangents
cannot be equal to 0, i.e. there is one to one correspondence between tangent rays at an arbitrary point $x \in \Sigma$
and connected components of $\Sigma \setminus \{x\}$. Moreover, if $d = 2$, then $\Sigma$ is a finite union of simple curves with
one-sided tangents continuous from the corresponding side.
\end{theorem}

In works concerning average distance minimizers the notion of \textit{blow up limits} is used.  Santambrogio and Tilli in~\cite{santambrogio2005blow} proved that for any average distance minimizer blow up sequence $\Sigma_\varepsilon \defeq \varepsilon^{-1} (\Sigma_a \cap B_\varepsilon(x) \text{ -- } x)$ with $x \in \Sigma_a$, converges in $B_1(0)$ (for the Hausdorff distance) to
some limit $\Sigma_0(x)$ when $\varepsilon \rightarrow 0$, and the limit is one of the following below (see Fig.~\ref{Fig:1} which is analogues to a picture from~\cite{lemenant2011presentationENG}), up to a rotation.

\begin{figure}[h]
        \centering
            \hfill \
   \begin{tikzpicture}[scale=1.9]
    \coordinate(O) at (0,0);
    \coordinate(x1) at (-1,0);
    \coordinate(x2) at (1,0);
    
   \draw [shift={(O)}]  plot[domain=0:6.5,variable=\t]({cos(\t r)},{sin(\t r)});
    
    \draw[very thick, blue] (x1) -- (x2);

   \fill (O) circle (1pt);
   \draw (O) node[above]{$x$};
   \node[align=center] at (0,-1.5) {a general $x$ has tangent line \\ $\psi({x})=0$};
 
   \end{tikzpicture}
     \hfill \
   \begin{tikzpicture}[scale=1.9]
   
    \coordinate(O) at (0,0);
    \coordinate(x1) at (-1,0);
    
   \draw [shift={(O)}]  plot[domain=0:6.5,variable=\t]({cos(\t r)},{sin(\t r)});
    
    \draw[very thick, blue] (O) -- (x1);
 \fill (O) circle (1pt);
   \draw (O) node[above]{$x$};
  \node[align=center] at (0,-1.5) {$x$ is a leaf \\$\psi({x})>0$};
\end{tikzpicture}
     \hfill \
   \begin{tikzpicture}[scale=1.9]
   
    \coordinate(O) at (0,0);
    \coordinate(x1) at ({cos(-0.2r)},{sin(-0.2r)});
    \coordinate(x2) at ({cos(3.3r)},{sin(3.3r)});
   
   \draw [shift={(O)}]  plot[domain=0:6.5,variable=\t]({cos(\t r)},{sin(\t r)});
    
    \draw[very thick, blue] (x1) -- (O) -- (x2);
 \fill (O) circle (1pt);
   \draw (O) node[above]{$x$};
       \node[align=center] at (0,-1.5)    {$x$ is a corner point\\ $\psi({x})>0$};

\end{tikzpicture}
    \hfill \
   \begin{tikzpicture}[scale=1.9]
   
    \coordinate(O) at (0,0);
    \coordinate(x1) at (1,0);
    \coordinate(x2) at ({cos(2.0943r)},{sin(2.0943r)});
    \coordinate(x3) at ({cos(4.18879r)},{sin(4.18879r)});
    
   \draw [shift={(O)}]  plot[domain=0:6.5,variable=\t]({cos(\t r)},{sin(\t r)});
    
    \draw[very thick, blue] (x1) -- (O) -- (x2);
    \draw[very thick, blue] (x3) -- (O);

   \fill (O) circle (1pt);
   \draw (O) node[left]{$x$};
   \node[align=center] at (0,-1.5)    {$x$ is a branching point\\ $\psi({x})=0$};
\end{tikzpicture}
    \hfill \ 
        \caption{All possible variants of tangent rays at any point of a maximal distance minimizer or blow up limits of an average distance minimizer}
        \label{Fig:1}
\end{figure}
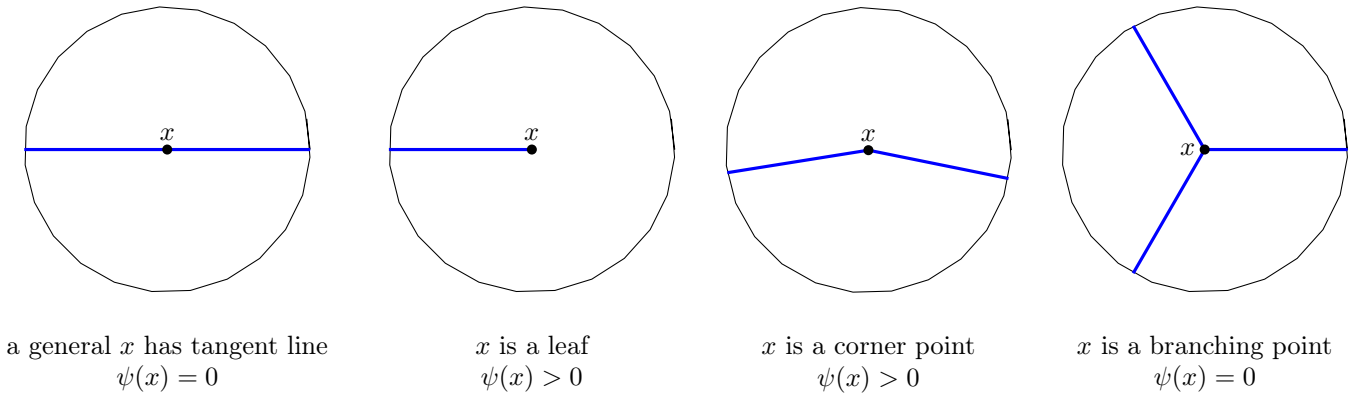

It is clear that for maximal distance minimizers blow up limits also exists and are more or less the same: $\Sigma_0$ can be a radius, a diameter, a union of two segments with the angle between the segments greater or equal to $2\pi/3$ or a center of a regular tripod. 
Thus a blow up limit always belongs to some 2-dimensional plane. 

At the second and third cases (id est when $\psi(x)>0$) the point $x$ has to be energetic; see the following definition. 

\begin{definition}
A point $x \in \Sigma$ is called \emph{energetic}, if for all $\rho>0$ one has
 	\[
		F_{M}(\Sigma \setminus B_{\rho}(x)) > F_{M}(\Sigma).
	\]
\end{definition}

Herewith if a point $x$ of a maximal distance minimizer $\Sigma$ is energetic then there exists such a point $y \in M$ (may be not unique) such that $\dist (x, y) = r$ and $B_{r}(y)\cap \Sigma=\emptyset$; such $y$ is called \textit{corresponding} to $x$.

If a point $x\in \Sigma$ is not energetic then in a sufficiently small neighbourhood it is a center of a regular tripod or a segment (and coincides there with its one-sided tangents). 

 A key object in all the study of the average distance problem is the pull-back measure of $\mu$ with respect to the projection onto $\Sigma_a$, where $\Sigma_a$ is a solution to the average distance minimizer problem. More precisely, if $\mu$ does not charge the Ridge set (which is defined as the set of all $x \in \R^d$ for which the minimum distance to $\Sigma_a$ is attained at more than one point) of $\Sigma_a$ (this is the case for instance when $\mu$ is absolutely continuous with respect to the Lebesgue
measure), then it is possible to choose a measurable selection of the projection multimap onto $\Sigma$, i.e. a map $\pi_\Sigma : M \rightarrow \Sigma$ such that $d(x, \Sigma)=d(x, \pi_{\Sigma_a})$ (this map is uniquely defined everywhere except the Ridge set). Then one can define the measure $\psi$ as
being $\psi(A):=\mu (\pi_{\Sigma_a}^{-1}(A))$, for any Borel set $A \subset M$.
In other words $\psi = \pi_{\Sigma_a} \sharp \mu$. 

For the maximal distance minimizers in $\R^d$ we can define measure $\psi$ in the similar way, but replace $M$ by $\partial B_r(\Sigma)$ and with $(n-1)$-dimensional Hausdorff measure as $\mu$ (or accordingly $\overline{B_r(\Sigma)}$ and $n$-dimensional Hausdorff measure). Thus Fig.~\ref{Fig:1} works both for maximal and average distance minimizers.

\subsection{Properties of branching points in \texorpdfstring{$\R^2$}{R2}}

Recall that by Theorem~\ref{th:GT} that for every planar compact set $M$ and a positive number $r$ a maximal distance minimizer can have only a finite number of points with $3$ tangent rays. 

In the plane it is also known (see~\cite{BS1}) that every average distance minimizer is topologically a tree composed of a finite union of simple curves joining with a number of $3$.

Every branching point of a planar maximal distance minimizer should be the center of a regular tripod.
If $x\in \Sigma \subset \R^2$ has $3$ tangent rays then there exists such a neighbourhood of $x$ in which the minimizer coincides with its tangent rays. Id est, there exists such $\varepsilon>0$ that $\Sigma \cap \overline{B_\varepsilon(x)}=[ax]\cup[bx]\cup [cx]$ where $\{a,b,c\}=\Sigma \cap \partial B_\varepsilon(x)$ and $\angle axb = \angle bxc = \angle cxa =2\pi/3$.
For planar average distance minimizers it is proved that any branching point admits such a neighbourhood in which three pieces of $\Sigma$ are $C^{1,1}$.

\subsection{Continuity of one-sided tangents in \texorpdfstring{$\R^2$}{R2}}
\begin{definition}\label{def:one-sided}
We say that the ray $ (ax] $ is a \emph{one-sided tangent} of a set $ \Gamma  \subset \mathbb{R}^d $ at a point $ x \in \Gamma $ if there exists a connected component $\Gamma_1$ of $\Gamma \setminus \{x\}$ such that $x \in \overline{\Gamma_1}$ and any sequence of points $x_k \in \Gamma_1$ with the property $x_k \rightarrow x$ satisfies $\angle x_kxa \rightarrow 0$.
In this case we also say that $(ax]$ is tangent to the connected component $\Gamma_1$.
\end{definition}
In the plane the continuity of one-sided tangents from the corresponding side holds (see~\cite{gordeev2022regularity}):
\begin{lemma}\label{zapred2}
Let $\Sigma \subset \mathbb{R}^2$ be a (local) maximal distance minimizer and let $x \in \Sigma$.
Let $\Sigma_1$ be a connected component of $\Sigma \setminus \{x\}$ with one-sided tangent $(ax]$ (it has to exist) and let $\bar x \in \Sigma_1$.
\begin{enumerate}
\item \label{zapred2:weak} For any one-sided tangent $(\bar a\bar x]$ of $\Sigma_1$ at $\bar x$ the equality $\angle((\bar a\bar x), (ax)) = o_{|\bar xx|}(1)$ holds.
\item Let $(\bar a\bar x]$ be a one-sided tangent at $\bar x$ of any connected component of $\Sigma_1 \setminus \{\bar x\}$ not containing $x$.
Then $\angle((\bar a\bar x], (ax]) = o_{|\bar xx|}(1)$.
\end{enumerate}
\end{lemma}

For planar average distance minimizers it is proved (see~\cite{lemenant2011presentationENG}) that away from branching points an average distance minimizer $\Sigma_a$ is locally at least as regular as the graph of a convex function, namely that the Right and Left tangent maps admit some Right and Left limits at every point and are semicontinuous. More precisely, for a given parametrization $\gamma$ of an injective Lipschitz arc $\Gamma \subset \Sigma_a$, by existence of blow up limits one can define the Left and Right tangent half-lines at every point $x \in \Gamma$ by 
\[
T_R(x) := x+\R^+.\lim_{h \rightarrow 0}\frac{\gamma(t_0+h)-\gamma(t_0)}{h}
\]
and
\[
T_L(x) := x+\R^+.\lim_{h \rightarrow 0}\frac{\gamma(t_0-h)-\gamma(t_0)}{h}.
\]
Then the following planar theorem for average distance minimizers holds.

\begin{theorem}[Lemenant,~2011~\cite{lemenant2011regularity}]
Let $\Gamma \subset \Sigma_a$ be an open injective Lipschitz arc. Then the Right and Left tangent maps $x \to T_R(x)$ and $x \to T_L(x)$ are semicontinuous, id est for every $y_0 \in \Gamma$ there holds $\lim_{y \to y_0; y<_\gamma y_0}T_L(y)=T_L(y_0)$ and $\lim_{y \to y_0; y>_\gamma y_0}T_R(y)=T_R(y_0)$.
In addition the limit from the other side exists and we have $\lim_{y \to y_0; y>_\gamma y_0}T_L(y)=T_R(y_0)$ and $\lim_{y \to y_0; y<_\gamma y_0}T_R(y)=T_L(y_0)$.
\end{theorem} 
An immediate consequence of the theorem is the following corollary: 
\begin{corollary}
Assume that $\Gamma \subset \Sigma$ is a relatively open subset of $\Sigma$ that contains no corner points nor branching points. Then $\Gamma$ is locally a $C^1$-regular curve.
\end{corollary}

\subsection{Planar example of infinite number of corner points}
Recall that each maximal distance minimizer in the plane is a finite union of simple curves. These curves should have continuous one-sided tangents but do not have to be $C^1$: there exists a minimizer with an infinite number of points without tangent lines. 
The following example is provided in~\cite{inverse2022}.

Fix positive reals $r$, $R$ and let $N$ be a large enough integer. Consider a sequence of points $\{a_i\}_{i=1}^\infty$ chosen from the circumference $\partial B_R(o)$ such that
$N \cdot |a_2a_1|=r$, 
\[
|a_{i+1}a_{i+2}| = \frac{1}{2}|a_ia_{i+1}|
\]
and $\angle a_ia_{i+1}a_{i+2} > \frac{\pi}{2}$ for every $i \in \mathbb{N}$ (see Fig.~\ref{Fig:cornerexample}). Let $a_\infty$ be the limit point of $\{a_i\}$. Finally, let $a_{\infty+1}$ be the point in the tangent line to $B_r(o)$ at $a_\infty$, such that 
\[
|a_\infty a_{\infty+1}| = r/N.
\]

We claim that polyline 
\[
\Sigma = \bigcup_{i=1}^{\infty} [a_ia_{i+1}]
\]
is a unique maximal distance minimizer for the following $M$.

Let $v_1 \in (a_1a_2]$ be such point that $|v_1a_1| = r$.
For $i\in \mathbb{N}\cup\{\infty\} \setminus \{1\}$ define $v_i$ as the point satisfying $|v_ia_i|=r$ and $\angle a_{i-1}a_iv_{i}=\angle a_{i+1}a_iv_{i}>\pi/2$. Define $v_\infty$ as the limit point of $\{v_i\}$.
Finally, let $v_{\infty + 1}$ be such point that $v_{\infty + 1}a_\infty \perp v_\infty a_\infty$ and $|v_{\infty+1} a_\infty| = r$. 
Clearly $M := \{v_i\}_{i=1}^{\infty+1}$ is a compact set.

\begin{figure}[h]
        \centering
            \hfill  
   \begin{tikzpicture}[scale=15]
    
    \coordinate(O) at (0,0);
    \coordinate(x1) at ({cos(2r)},{sin(2r)});
    \coordinate(x2) at ({cos(1.6r)},{sin(1.6r)});
    \coordinate(x3) at ({cos(1.4r)},{sin(1.4r)});
    \coordinate(x4) at ({cos(1.3r)},{sin(1.3r)});
    \coordinate(x5) at ({cos(1.25r)},{sin(1.25r)});
    \coordinate(x6) at ({cos(1.225r)},{sin(1.225r)});
    \coordinate(xN) at ({cos(1.15r)},{sin(1.15r)});

    \foreach \x in{1,2,...,6}{
    \coordinate(y\x) at ($(x\x)!-0.1!(O)$);
    }    
    \coordinate(yy) at ($(xN)!-0.1!(O)$);
    \coordinate(yN) at ($(xN)!-2.2!(x6)$);
    \coordinate(xNN) at ($(xN)!-0.8!(x6)$);
    \coordinate(y0) at ($(x1)!-0.28!(x2)$);

    \draw [shift={(O)}]  plot[domain=1:2.141,variable=\t]({cos(\t r)},{sin(\t r)});
    
    \foreach \x in{2,3,...,6}{
    \draw[dotted] (x\x) -- (y\x);
    }
    \draw[dotted] (x1) -- (y0);
    \draw[dotted] (xN) -- (yN);
    \draw[dotted] (xN) -- (yy);

    \foreach \x in{2,3,...,5}{
            \draw (y\x) node[above]{$v_\x$};
            \edef\y{\x}
            \pgfmathparse{\y+1}
            \edef\y{\pgfmathresult}
        }

    \foreach \x in{2,3,...,5}{
            \draw [dotted, shift={(y\x)}]  plot[domain={3.4-\x*0.1}:{6.3-\x*0.05},variable=\t]({cos(\t r)*0.1},{sin(\t r)*0.1});
        }

    \foreach \x in{1,2,...,5}{
            \draw (x\x) node[above]{$a_\x$};
            \edef\y{\x}
            \pgfmathparse{\y+1}
            \edef\y{\pgfmathresult}
            \draw [ultra thick, blue] (x\x) -- (x\y);
        }

    \draw [ultra thick, blue] (xN) -- (xNN);

    \draw (x6) node[above right]{$a_6$};
    \draw (xN) node[below]{$a_{\infty}$};
    \draw (y6) node[above right]{$v_6$};
    \draw (y0) node[above right]{$v_1$};
    \draw (yy) node[above right]{$v_\infty$};
    \draw (yN) node[above right]{$v_{\infty+1}$};
    \draw (xNN) node[above right]{$a_{\infty+1}$};
    
    \draw [line width=2pt, dash dot, blue] (x6) -- (xN);
    \draw [line width=2pt, dash dot] (y6) -- (yy);

    \fill [blue] (xN) circle (0.2pt);
    \fill[blue] (xNN) circle (0.2pt);
    
    \foreach \x in{1,2,...,6}{
        \fill [blue] (x\x) circle (0.2pt);
    }
    \foreach \x in{2,...,6}{
        \fill (y\x) circle (0.2pt);
    }
    \fill (y0) circle (0.2pt);
    \fill (yN) circle (0.2pt);
    \fill (yy) circle (0.2pt);
    
\end{tikzpicture}
    \hfill \ 
        \caption{The example of a minimizer with infinite number of corner points}
        \label{Fig:cornerexample}
\end{figure}
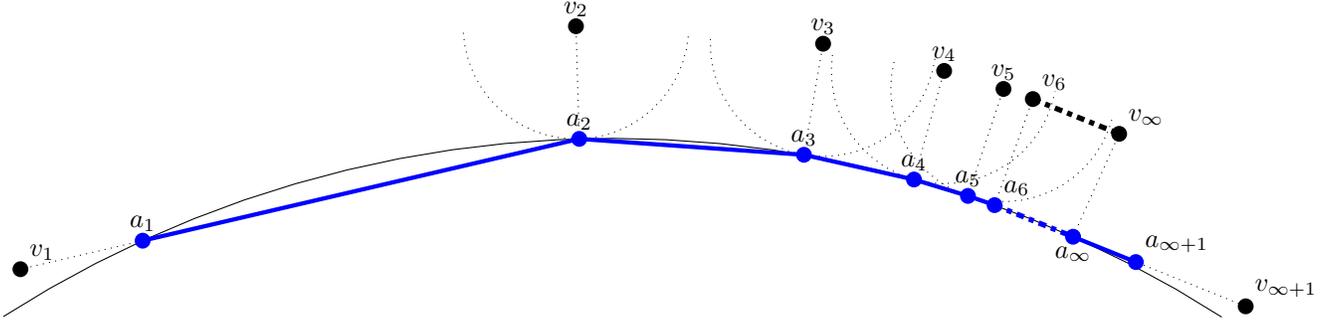

\begin{theorem}[Basok--Cherkashin--Teplitskaya, 2022~\cite{inverse2022}] \label{theocornerpoints}
Let $\Sigma$ and $M$ be defined above. Then $\Sigma$ is a unique maximal distance minimizer for $M$.
\end{theorem}

\subsection{Every \texorpdfstring{$C^{1,1}$}{C11}-smooth simple curve is a minimizer}

For planar average distance minimizers Tilli proved in~\cite{tilli2010some} that any simple $C^{1,1}$-curve is a minimizer for a proper input. 
Paper~\cite{inverse2022} generalizes Tilli's result on $d$-dimensional space. The same statement with a similar but much simpler explanation is true for maximal distance minimizers.

\begin{theorem}[Basok--Cherkashin--Teplitskaya, 2022~\cite{inverse2022}] \label{Th:C11isaminimizer}
Let $\gamma \subset \mathbb{R}^d$ be a simple $C^{1,1}$-curve. Then $\gamma$ is a maximal distance minimizer for a small enough $r$ and $M=\overline{B_r(\gamma)}$.
\end{theorem}

\section{Explicit examples for maximal distance minimizers}

Recall that Theorems~\ref{theocornerpoints} and~\ref{Th:C11isaminimizer} provide explicit examples; however they are obtained by ``reverse engineering'': the input $M$ is constructed in a way to give the minimizer property to a desired $\Sigma$.
This section is devoted to known explicit results. 

\subsection{Simple examples. Finite number of points and \texorpdfstring{$r$}{r}-neighbourhood. Inverse minimizers}
\label{sec:Sttree}

Here we consider Problem~\ref{TheDualProblem} in a case when $M$ is a finite set. Then it is closely related with the following Steiner problem.

\begin{problem}\label{Problem1}
For a given finite set $P = \{x_1,\dots ,x_n\} \subset \mathbb{R}^d$ to find a connected set $\St(P)$ with the minimal length (one-dimensional Hausdorff measure) containing $P$.
\end{problem}

A solution $\St$ to Problem~\ref{Problem1} is called a \textit{Steiner tree}. It can be represented as a plane graph, such that its set of vertices contains $P$, and all its edges
are line segments. This graph is connected and does not contain cycles, i.e. is a tree, which explains the naming of $S$. It is
known that the maximum vertex degree of this graph is 3.
Only vertices $x_i$ can have degree $1$ or $2$, all the other vertices have degree $3$ and are called \textit{Steiner points} while the vertices $x_i$ are called \textit{terminals}. 
Vertices of the degree $3$ are called \textit{branching points}. The angle between any two adjacent edges of $\St$ is at least $2\pi/3$.
That means that for a branching point the angle between any two segments incident to it is exactly $2\pi/3$, and these three segments belong to the same 2-dimensional plane. The number of Steiner points in $\St$ does not exceed $n-2$. A Steiner tree with exactly $2n-2$ vertices is called \textit{full}. Every terminal point of a full Steiner tree has degree one. 
Full trees are \textit{indecomposable} in the sense that a full tree cannot be represented as a union of Steiner trees for non-empty $P_1$ and $P_2$ such that $P_1 \cup P_2 = P$.

We define a \textit{locally minimal tree} for $P$ as a connected compact acyclic set $S$, which contains $P$, and such that for any 
$x \in S \setminus P$ there is a neighborhood $U \ni x$ such that $S \cap U$ coincides with the Steiner tree on the set of points $(S \cap \partial U) \cup (P \cap U)$. 
A locally minimal tree retains the following properties of a Steiner tree: it is the union of a finite set of segments; the angle
between any two adjacent segments is at least $2\pi/3$.

Proofs of the listed properties and more information on Steiner trees could be found in~\cite{gilbert1968steiner,hwang1992steiner}.

Any maximal distance minimizer for any finite set $M \subset \mathbb{R}^d$ is a finite union of at most $2\sharp M-3$ segments. In this case maximal distance minimization problem comes down to connecting $r$-neighborhoods of all the points from $M$. If $\overline{B_r(a)}$ are disjoint for every $a \in M$ then a maximal distance minimizer is a Steiner tree connecting some points from $\partial B_r(a)$, $a \in M$.

The following observations and statements of this paragraph are from the paper~\cite{inverse2022}.
\begin{remark}
\begin{itemize}
    \item[(i)] Let $\Sigma$ be a maximal distance minimizer for some $M$ and $r>0$. Then $\Sigma$ is a maximal distance minimizer for $\overline{B_r(\Sigma)}$ and $r$.
    \item[(ii)] Let $\Sigma$ be a minimizer for $\overline{B_r(\Sigma)}$ and $r>0$. Then $\Sigma$ is a minimizer for $\overline{B_{r'}(\Sigma)}$ and $r'$, where $0 < r' < r$.
\end{itemize}
\label{Obs:MisBrSigma}
\end{remark}

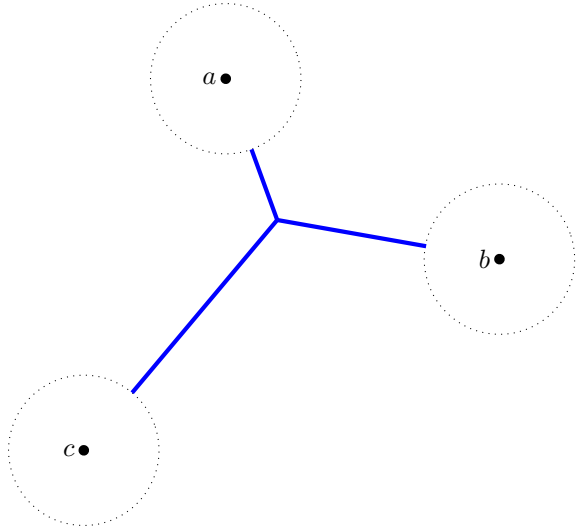
\begin{figure}[h]
	\centering
         \hfill
    \begin{tikzpicture}[rotate=20]

    \coordinate (T) at (0,0);
    \coordinate (A) at (0,2);
    \coordinate (B) at ({3*cos(-30)},{3*sin(-30)});
    \coordinate (C) at ({4*cos(-150)},{4*sin(-150)});
    
    \draw[dotted] (A) circle (1cm);
    \draw[dotted] (B) circle (1cm);
    \draw[dotted] (C) circle (1cm);
    
    \draw[ultra thick, blue] (T)--($(A)!+0.50!(T)$);
    \draw[ultra thick, blue] (T)--($(B)!+0.33!(T)$);
    \draw[ultra thick, blue] (T)--($(C)!+0.25!(T)$);
    
    \fill (A) circle (2pt);
    \fill (B) circle (2pt);
    \fill (C) circle (2pt);

    \draw (A) node[left]{$a$};
    \draw (B) node[left]{$b$};
    \draw (C) node[left]{$c$};
    
    \end{tikzpicture}
     \hfill
    \caption{A maximal distance minimizer for a certain $3$-point set $M = \{a,b,c\}$}
    \label{tripod}
\end{figure}

A \textit{topology} $T$ of a labelled Steiner tree (or a labelled locally minimal tree) $\St$ is the corresponding abstract graph with labelled terminals and unlabelled Steiner points.

\begin{theorem}[Basok--Cherkashin--Teplitskaya,~2022~\cite{inverse2022}] \label{theo:finiteinverse}
Let $\St$ be a Steiner tree for a labelled set of terminals $A = (a_1,\dots, a_n)$, $a_i \in \mathbb{R}^d$ such that every Steiner tree for an $n$-tuple in the closed $2r$-neighbourhood of $A$ (with respect to $\rho$) has the same topology as $\St$ for some positive $r$. Then $\St$ is an $r$-minimizer for an $n$-tuple $M$.
\end{theorem}

In the plane a Steiner tree for a random input is unique with unit probability, see~\cite{basok2018uniqueness}. Also in the plane we have a general inverse statement to Theorem~\ref{theo:finiteinverse}. 

\begin{proposition}[Basok--Cherkashin--Teplitskaya,~2022~\cite{inverse2022}]
Suppose that $\St$ is a full Steiner tree for terminals $a_1,\dots, a_n \in \mathbb{R}^2$, which is not unique. 
Then $\St$ can not be a minimizer for $M$ being an $n$-tuple of points. 
\label{prop:notaminimizer}
\end{proposition}

To illustrate Proposition~\ref{prop:notaminimizer} consider a square $a_1a_2a_3a_4$. There are two Steiner trees for $a_1,a_2,a_3,a_4$ (see the left-hand side of Fig.~\ref{picnotunique}), let us pick the solid one. The right-hand side of Fig.~\ref{picnotunique} shows that an $r$-minimizer for every positive $r$ has the topology of the dotted Steiner tree.

In all known examples a $\St$ with $n$ terminals is an $r$-minimizer for a set $M$ of $n$ points and a small enough positive $r$ if and only if $\St$ in the unique Steiner tree for its terminals. So the planar case of several non-full solutions is open, and also it is interesting to derive any analogue of Proposition~\ref{prop:notaminimizer} for $d > 2$.

\begin{figure}[h]
	\centering
        \hfill  
   \begin{tikzpicture}
    \def\r{1.5cm}
    \draw[ultra thick, blue]
        (-\r, \r) coordinate(x1) node[black, above right]{$1$} --++ (-60:{\r/cos(30)}) coordinate (x5);
    \draw[ultra thick, blue]
        (\r,\r) coordinate(x2) node[black, above left]{$2$} --++ (-120:{\r/cos(30)}) coordinate (x6);
    \draw[ultra thick, blue]
        (\r, -\r) coordinate(x3) node[black, below left]{$3$} --++ (120:{\r/cos(30)});
    \draw[ultra thick, blue]
        (-\r,-\r) coordinate(x4) node[black, below right]{$4$} --++ (60:{\r/cos(30)});
    \draw[ultra thick, blue]
        (x5) -- (x6);
        
    \draw [dotted,white] (x1)  --++ (-60:{-\r/2}) coordinate (x11);
    \draw [dotted,white] (x2)  --++ (-120:{-\r/2}) coordinate (x12);
    \draw [dotted,white] (x3)  --++ (120:{-\r/2}) coordinate (x13);
    \draw [dotted,white] (x4)  --++ (60:{-\r/2}) coordinate (x14);
    
    \draw [dashed,shift={(x11)},white]  plot[domain=0:6.5,variable=\t]({cos(\t r)/1.35},{sin(\t r)/1.35});
    \draw [dashed,shift={(x12)},white]  plot[domain=0:6.5,variable=\t]({cos(\t r)/1.35},{sin(\t r)/1.35});
    \draw [dashed,shift={(x13)},white]  plot[domain=0:6.5,variable=\t]({cos(\t r)/1.35},{sin(\t r)/1.35});
    \draw [dashed,shift={(x14)},white]  plot[domain=0:6.5,variable=\t]({cos(\t r)/1.35},{sin(\t r)/1.35});

    \draw[ultra thick, blue, dashed]
        (x1)  --++ (-30:{\r/cos(30)}) coordinate (x7);
    \draw[ultra thick, blue, dashed]
        (x2)  --++ (-150:{\r/cos(30)});
    \draw[ultra thick, blue, dashed]
        (x3)  --++ (150:{\r/cos(30)}) coordinate (x8);
    \draw[ultra thick, blue, dashed]
        (x4) --++ (30:{\r/cos(30)});
    \draw[ultra thick, blue, dashed]
        (x7) -- (x8);
    \foreach \x in{1,2,...,8}{
        \fill (x\x) circle (2pt);
    }
\end{tikzpicture}
\hspace{2cm}
\begin{tikzpicture}
    \def\r{1.5cm}
    \draw[ultra thick, blue]
        (-\r, \r) coordinate(x1) node[black, below left]{$1$} --++ (-60:{\r/cos(30)}) coordinate (x5);
    \draw[ultra thick, blue]
        (\r,\r) coordinate(x2) node[black, below right]{$2$} --++ (-120:{\r/cos(30)}) coordinate (x6);
    \draw[ultra thick, blue]
        (\r, -\r) coordinate(x3) node[black, above right]{$3$} --++ (120:{\r/cos(30)});
    \draw[ultra thick, blue]
        (-\r,-\r) coordinate(x4) node[black, above left]{$4$} --++ (60:{\r/cos(30)});
    \draw[ultra thick, blue]
        (x5) -- (x6);
        
    \draw [dotted] (x1)  --++ (-60:{-\r/2}) coordinate (x11);
    \draw [dotted] (x2)  --++ (-120:{-\r/2}) coordinate (x12);
    \draw [dotted] (x3)  --++ (120:{-\r/2}) coordinate (x13);
    \draw [dotted] (x4)  --++ (60:{-\r/2}) coordinate (x14);
    
    \draw [dashed,shift={(x11)}]  plot[domain=0:6.5,variable=\t]({cos(\t r)/1.35},{sin(\t r)/1.35});
    \draw [dashed,shift={(x12)}]  plot[domain=0:6.5,variable=\t]({cos(\t r)/1.35},{sin(\t r)/1.35});
    \draw [dashed,shift={(x13)}]  plot[domain=0:6.5,variable=\t]({cos(\t r)/1.35},{sin(\t r)/1.35});
    \draw [dashed,shift={(x14)}]  plot[domain=0:6.5,variable=\t]({cos(\t r)/1.35},{sin(\t r)/1.35});

    \draw [dotted] (x11)  --++ (-30:{\r/2}) coordinate (x21);
    \draw [dotted] (x12)  --++ (-150:{\r/2}) coordinate (x22);
    \draw [dotted] (x13)  --++ (150:{\r/2}) coordinate (x23);
    \draw [dotted] (x14)  --++ (30:{\r/2}) coordinate (x24);

    \draw[ultra thick, blue, dashed]
        (x21)  --++ (-30:{\r*0.94}) coordinate (x7);
    \draw[ultra thick, blue, dashed]
        (x22)  --++ (-150:{\r*0.94});
    \draw[ultra thick, blue, dashed]
        (x23)  --++ (150:{\r*0.94}) coordinate (x8);
    \draw[ultra thick, blue, dashed]
        (x24) --++ (30:{\r*0.94});
    \draw[ultra thick, blue, dashed]
        (x7) -- (x8);
    
    \foreach \x in{1,2,...,8}{
        \fill (x\x) circle (2pt);
        }
    \foreach \x in{1,2,...,4}{
        \fill (x2\x) circle (2pt);
    }
\end{tikzpicture}
    \hfill \ 
    \caption{An example to Proposition~\ref{prop:notaminimizer}}
    \label{picnotunique}
\end{figure}
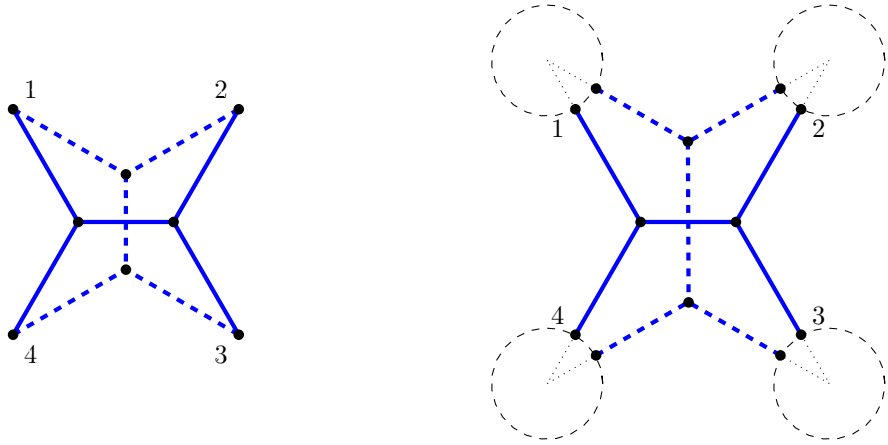

\subsection{Circle. Curves with big radius of curvature}

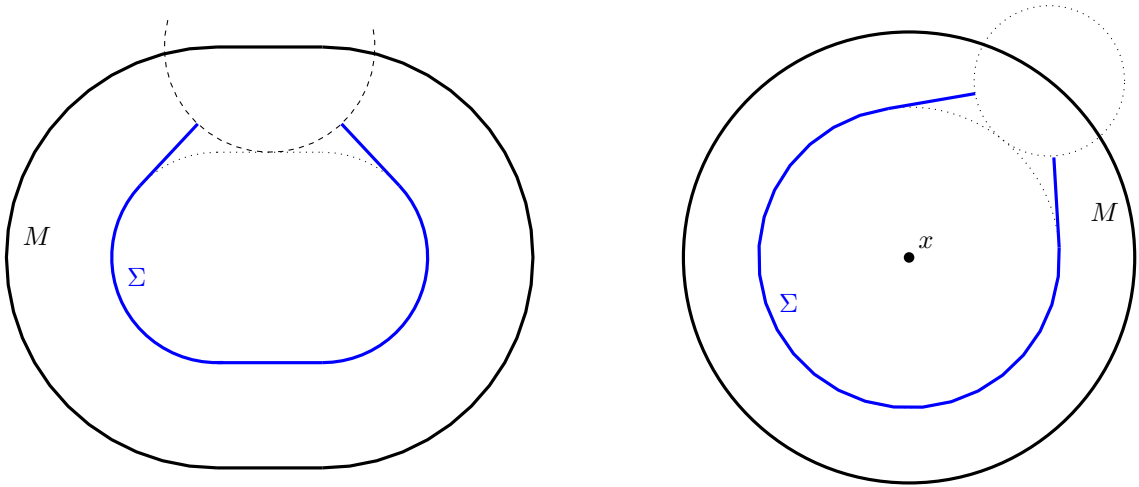
\begin{figure}[h]
	\centering
         \hfill
    \begin{tikzpicture}

    \begin{scope}[scale=1.4]

    \draw [very thick] (-3.,2.)-- (-2.,2.);
    \draw [very thick] (-2.,-2.)-- (-3.,-2.);
    \draw [very thick, color=blue] (-3.,-1.)-- (-2.,-1.);
    \draw [dotted] (-3.,1.)-- (-2.,1.);
    \draw [shift={(-3.,0.)},very thick]  plot[domain=1.5707963267948966:4.71238898038469,variable=\t]({2.*cos(\t r)},{2.*sin(\t r)});
    \draw [shift={(-3.,0.)},dotted]  plot[domain=1.5707963267948966:4.71238898038469,variable=\t]({cos(\t r)},{sin(\t r)});
    \draw [shift={(-2.,0.)},dotted]  plot[domain=-1.5707963267948966:1.5707963267948966,variable=\t]({cos(\t r)},{sin(\t r)});
    \draw [shift={(-2.,0.)},very thick]  plot[domain=-1.5707963267948966:1.5707963267948966,variable=\t]({2.*cos(\t r)},{2.*sin(\t r)});
    \draw [shift={(-2.5008100281931047,2.)},dash pattern=on 2pt off 2pt] plot[domain=-3.4:0.2,variable=\t]({cos(\t r)},{sin(\t r)});
    \draw [very thick,color=blue] (-3.7309420990521054,0.6824394829091469)-- (-3.1832495111022516,1.2690579009478955);
    \draw [very thick,color=blue] (-1.8178907146013357,1.2695061868000874)-- (-1.2695061868000874,0.6829193135917684);
    \draw [shift={(-2.,0.)},very thick,color=blue]  plot[domain=-1.5707963267948966:0.7517515639553677,variable=\t]({cos(\t r)},{sin(\t r)});
    \draw [shift={(-3.,0.)},very thick,color=blue]  plot[domain=2.390497746093128:4.771687465439039,variable=\t]({cos(\t r)},{sin(\t r)});
    \draw (-4.7089822121844977,0.2025724503290166) node {$M$};

    \draw[color=blue] (-3.763965847825129,-0.1818662773850787) node {$\Sigma$};

    \end{scope}
 
    \begin{scope}[shift = ({5,0})]
        
    \coordinate(O) at (0,0);
    \coordinate(x1) at ({2*cos(1.71 r)}, {2*sin(1.71 r)});
    \coordinate(x2) at ({2*cos(6.35 r)}, {2*sin(6.35 r)});
    \coordinate(y)  at ({3*cos(0.90 r)}, {3*sin(0.90 r)});
    
    \draw [very thick] (O) circle (3cm);
    \draw [dotted] (O) circle (2cm);
    \draw [dotted] (y) circle (1cm);
    
    \draw [very thick,color=blue] (x1)-- ($(y)!+0.46!(x1)$);
    \draw [very thick,color=blue] (x2)-- ($(y)!+0.46!(x2)$);
    \draw [shift={(0.,0.)},very thick,color=blue]  plot[domain=1.71:6.35,variable=\t]({2*cos(\t r)},{2*sin(\t r)});


    \fill (O) circle (2pt);
    \draw (O) node [above right] {$x$};

    \draw (2.6,0.6) node {$M$};
    
    \draw[color=blue] (-1.6,-0.6) node {$\Sigma$};
    
    \end{scope}

    \end{tikzpicture}
     \hfill
    \caption{A minimizer for a convex closed planar curve $M$ with the radius of curvature at least $5r$ at every point, so-called \textit{horseshoe} (left). A minimizer for $M = \partial B_R(x)$, where $R > 4.98r$ (right)}
    \label{horseshoe}
\end{figure}

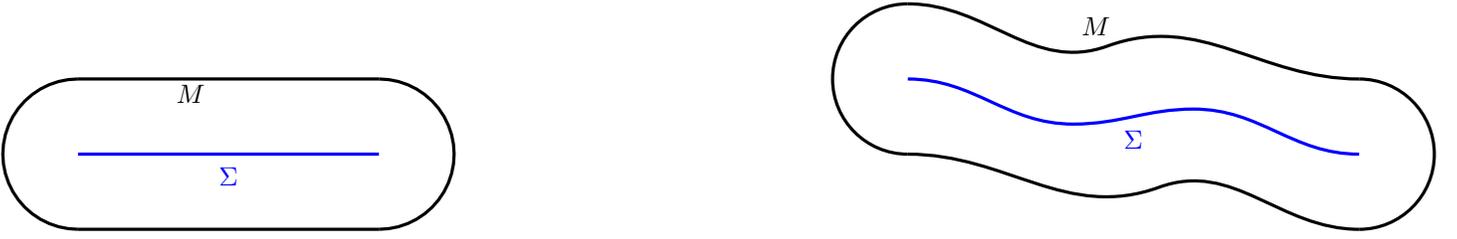
\begin{figure}[h]
	\centering
    \begin{tikzpicture}

    \draw [very thick] (-2, 1) -- (2, 1);
    \draw [very thick] (-2,-1) -- (2,-1);
    \draw [very thick, color=blue] (-2,0)-- (2,0);

    \draw [shift={(-2.,0.)},very thick]  plot[domain=1.5707963267948966:4.71238898038469,variable=\t]({cos(\t r)},{sin(\t r)});
    \draw [shift={(2.,0.)},very thick]  plot[domain=-1.5707963267948966:1.5707963267948966,variable=\t]({cos(\t r)},{sin(\t r)});

    \draw (-0.5,0.8) node {$M$};
    \draw[color=blue] (0,-0.3) node {$\Sigma$};
    \end{tikzpicture}
     \hfill
    \begin{tikzpicture}[scale=1]

        \coordinate (O) at (0, 0);
        \coordinate (A) at (-3, 0.5);
        \coordinate (B) at ( 3,-0.5);

        \coordinate (C) at (-1.9, 0.2);
        \coordinate (D) at ( 1.9,-0.2);
        
        \coordinate (M) at (-0.8,-0.1);
        \coordinate (N) at ( 0.8, 0.1);

        \coordinate (x1) at ({cos(110)}, {sin(70)});
        \coordinate (x2) at ({-cos(110)}, {-sin(70)});

        \coordinate (a1) at (-3, 1.5);
        \coordinate (a2) at (-3, -0.5);

        \coordinate (b1) at (3, 0.5);
        \coordinate (b2) at (3,-1.5);

        \def\radius{1}

        \draw[very thick, blue] (A) to [out=0, in=-180] (M) to [out=0, in=-180] (N) to[out=0, in=180] (B);
        \draw[very thick] (a1) to [out=0, in=-160] (x1) to[out=20, in=180] (b1);
        \draw[very thick] (a2) to [out=0, in=-160] (x2) to[out=20, in=180] (b2);

        \draw[very thick]  (a1) arc (90:270:1);
        \draw[very thick]  (b1) arc (90:-90:1);
        
        \draw (-0.5,1.2) node {$M$};
    \draw[color=blue] (0,-0.3) node {$\Sigma$};

    \end{tikzpicture}
    \caption{$M$ is $r$-neighbourhood for a sufficiently smooth curve $\Sigma$ and small enough $r >0$}
    \label{fig:curves}
\end{figure}

Let $M_r$ be the inner part of the boundary of $B_r(M)$.

\begin{theorem}[Cherkashin--Teplitskaya, 2018~\cite{cherkashin2018horseshoe}]
Let $r$ be a positive real, $M$ be a convex closed curve with the radius of curvature at least $5r$ at every point, $\Sigma$ be an arbitrary minimizer for $M$. Then $\Sigma$ is a union of an arc of $M_r$ and two segments that are tangent to $M_r$ at the ends of the arc (so-called horseshoe, see Fig.~\ref{horseshoe}). 
In the case when $M$ is a circumference with the radius $R$, the condition $R > 4.98r$ is enough.
\label{horseshoeT}
\end{theorem}

Also Theorem~\ref{horseshoeT} admits a corollary on local minimizers in the sense of Definition~\ref{def:local}.

\begin{corollary}[Cherkashin--Teplitskaya, 2018~\cite{cherkashin2018horseshoe}] \label{coroll:horseshoe}
    Let $\hat \Sigma$ be a local minimizer for some closed convex curve $M$ with minimal radius of curvature $R > 5r$.
    Then if $\hat \Sigma$ is not a horseshoe, one has $\H (\hat \Sigma) - \H (\Sigma) \geq (R-5r)/2$, 
    where $\Sigma$ is an arbitrary (global) minimizer. 
\end{corollary}

Miranda, Paolini and Stepanov~\cite{miranda2006one} conjectured that all the minimizers for a circumference of radius $R > r$ are horseshoes. Theorem~\ref{horseshoeT} solves this conjecture with the assumption $R > 4.98r$; for $4.98r \geq R > r$ the conjecture remains open.

\subsection{Rectangle}

\begin{theorem}[Cherkashin--Gordeev--Strukov--Teplitskaya, 2021~\cite{cherkashin2021maximal}]
Let $M = a_1a_2a_3a_4$ be a rectangle. Then there is a positive number $r_0(M)$ such that for any positive $r < r_0(M)$ every minimizer of the maximum distance functional has a topology of 21 segments, shown on the leftmost side of Fig.~\ref{rectangle}. The middle part of the figure shows an enlarged fragment of the minimizer in the vicinity of point $a_1$; the marked angles are equal to $\frac{2\pi}{3}$. The rightmost side of the figure shows an even more enlarged fragment of the minimizer in the vicinity of $a_1$.

Any minimizer of the maximum distance functional has length $\Per(M) - cr + o(r)$, where $\Per(M)$ is the perimeter of the rectangle $M$, and $c$ is a constant approximately equal to 8.473981.
\label{rectangleT}
\end{theorem}

In fact, every maximal distance minimizer is very close (in the sense of Hausdorff distance) to the one depicted in the picture.
\begin{figure}[h]
    \centering

    \begin{tikzpicture}
    
    \begin{scope}[scale=0.30]
        
        \def\w{8}
        \def\h{4.5}
        
        \def\t{.85}
        \def\c{}
        
        \foreach\x in {-1,1} {
            \foreach\y in {-1,1} {
                \draw[blue, thick]  ({\x*(-\w+1 - .292893218813455)},
                        {\y*(-\h+1 - .2928932188134548}) -- 
                       ({\x*(-\w+2 - \t)},
                        {\y*(-\h+2 - \t)});
                \draw[blue, thick]  ({\x*(-\w+2 - \t)},
                        {\y*(-\h+2 - \t)}) -- 
                       ({\x*(-\w+1)}, 
                        {\y*(-\h+2 -.5857864376269108)});
                \draw[blue, thick]  ({\x*(-\w+2 - \t)},
                        {\y*(-\h+2 - \t)}) -- 
                       ({\x*(-\w+2 - .5857864376269095)},
                        {\y*(-\h+1)});
            }
        }
        \draw[blue, thick]  (-\w+2 - .5857864376269095,-\h+1) -- 
               ( 0,-\h);
        \draw[blue, thick]  ( 0,-\h) -- 
               ( \w-2 + .5857864376269095,-\h+1);
        \draw[blue, thick]  (-\w+1, -\h+2 - .5857864376269108) -- 
               (-\w+1,  \h-2 + .5857864376269108);
        \draw[blue, thick]  (-\w+2 - .5857864376269095, \h-1) -- 
               ( \w-2 + .5857864376269095, \h-1);
        \draw[blue, thick]  ( \w-1, -\h+2 - .5857864376269108) -- 
               ( \w-1,  \h-2 + .5857864376269108);
               
        \fill[white] (1, -\h) arc (0:180:1);
        
        \draw [very thick] (-\w,-\h) rectangle (\w,\h);
        
        \draw (-\w, -\h) node[below left]{$A_1$};
        \draw (-\w, \h) node[above left]{$A_2$};
        \draw (\w, \h) node[above right]{$A_3$};
        \draw (\w, -\h) node[below right]{$A_4$};

    \end{scope}    
    
    \begin{scope}[scale=1.6,shift=({3,-1})]
    
        \coordinate (V) at (1.10837, 1.10837);
        \coordinate (Q) at (.7071, .7071);
        \coordinate (C) at (.73, .73);
        
        \draw[dashed] (0,0) -- (Q);
        \draw[thick, dotted] (0, 1) arc (90:0:1);
        \draw[thick, dotted] (.414, 0) arc (180:55:1);
        \draw[dashed] (1.414, 0) -- (1.544739754593147, 0.9914448613738104) node[pos=0.5, left]{$r$};
        \draw[thick, dotted] (0, .414) arc (-90:35:1);
        \draw[dashed] (0, 1.414) -- (0.9914448613738104,1.544739754593147) node[pos=0.5,below]{$r$};
        
        \draw [very thick] (2,0) -- (0,0) -- (0,2);
        \draw (0, 0) node[below left]{$A_1$};
        \draw [very thick, blue] (0.9914448613738104,2)-- (0.9914448613738104,1.544739754593147);
        \draw [very thick, blue] (V)-- (0.9914448613738104,1.544739754593147);
        \draw [very thick, blue] (V)-- (1.5447397545931463,0.9914448613738106);
        \draw [very thick, blue] (1.5447397545931463,0.9914448613738106)-- (2,0.9914448613738102);
        \draw [very thick, blue] (Q)-- (V);    

        \def\rr{0.12}
        \draw[shift={(V)}] 
            (0.966*\rr ,-0.2588*\rr) arc(-15:105:\rr);
        \def\rr{0.09}
        \draw[shift={(V)}] 
            (0.966*\rr ,-0.2588*\rr) arc(-15:105:\rr);
        \def\rr{0.096}
        \draw[shift={(V)}] 
            (0.966*\rr ,-0.2588*\rr) arc(-15:-135:\rr);
        \def\rr{0.072}
        \draw[shift={(V)}] 
            (0.966*\rr ,-0.2588*\rr) arc(-15:-135:\rr);
        \def\rr{0.112}
        \draw[shift={(V)}] 
            (-0.2588*\rr ,0.966*\rr) arc(105:225:\rr);
        \def\rr{0.082}
        \draw[shift={(V)}] 
            (-0.2588*\rr ,0.966*\rr) arc(105:225:\rr);
        
\node[above] at (1.644739754593147, 0.9914448613738104) {\small $\approx \frac{11\pi}{12}$};
        
        \def\ll{0.051}
        \def\llll{0.2}
        \draw[line width = .7pt] (C) circle (\ll);
        \draw[shift={(C)}, line width = 2pt] (\ll * .866, \ll * -.5) -- (\llll * 0.866, \llll * -.5);
    \end{scope}

    \begin{scope}[scale=35, shift=({-0.4,-0.732})]
        \coordinate (V) at (.7645532062, .76593);
        \coordinate (Q2) at (.723714, .725155);
        \coordinate (Q1) at (.707224, .706989);
        \coordinate (C) at (.73, .73);

        \draw[very thick, blue] (V) -- (Q2) -- (Q1);
        \fill[blue] (Q2) circle (0.0015);
        
        \draw[shift={(Q2)}, rotate=45] (0.0075, 0) arc (0 : -177 : 0.0075) node[right, pos=0.3]{$\approx 0.98\, \pi$};
        
        \fill[black, shift={(C)}, rotate=60] (-0.01, -0.048) -- (-0.013, -0.0715) -- (0.01, -0.06) -- (0.01, -0.048) -- cycle;
        \draw[black, line width=3] (C) circle (0.05);
        
        \fill[black, shift={(C)}, rotate=60] (0.0171, -0.046) arc (-70 : -110 : 0.05) -- (-0.0171, -0.049) arc (-120 : -60 : 0.0342) -- cycle;
        \end{scope}    
    
    \end{tikzpicture}

    \caption{The minimizer for a rectangle $M$ with $r < r_0(M)$.}

    \label{rectangle}
\end{figure}
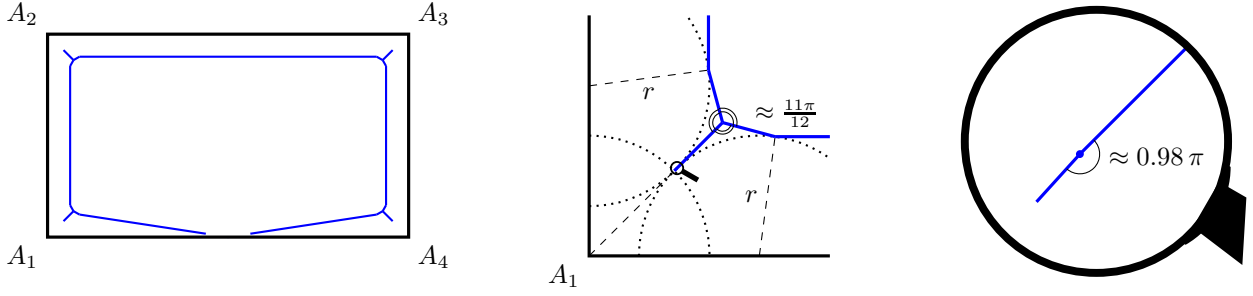    

\section{Tools}

\subsection{Energetic points}

For the planar problem the notion of energetic points (which is also defined in $\R^d$) is very useful.

Recall that a point $x \in \Sigma$ is called \emph{energetic}, if for all $\rho>0$ one has $F_{M}(\Sigma \setminus B_{\rho}(x)) > F_{M}(\Sigma)$. The set of all energetic points of $\Sigma$ is denoted by $G_\Sigma$.
Each minimizer $\Sigma$ can be split into three disjoint subsets:
        \[
        \Sigma=E_{\Sigma}\sqcup\X_{\Sigma}\sqcup\S_{\Sigma},
        \]
        where $X_{\Sigma}\subset G_\Sigma$ 
is the set of isolated energetic points 
(i.e. every $x\in X_\Sigma$ is energetic and there is a $\rho>0$ possibly depending on $x$ such that $B_{\rho}(x)\cap G_{\Sigma}=\{x\}$), $E_{\Sigma} := G_{\Sigma}\setminus X_{\Sigma}$ is the set of non-isolated energetic points and $S_\Sigma \defeq \Sigma \setminus G_\Sigma$ is the set of non-energetic points also called the \textit{Steiner part}.

Note that it is possible for a (local) minimizer in $\mathbb{R}^d$, $d>2$ to have no non-energetic points at all.
Moreover, in some sense, any (local) minimizer does not have non-energetic points in an embedding into a larger dimension:

\begin{example}\label{rem:all-energ}
Let $\Sigma$ be a (local) minimizer for a compact set $M \subset \mathbb{R}^d$ and $r > 0$.
Then $\bar \Sigma \defeq \Sigma \times \{0\} \subset \mathbb{R}^{d + 1}$ is a (local) minimizer for $\bar M = (M \times \{0\}) \cup (\Sigma \times \{r\}) \subset \mathbb{R}^{d + 1}$ and $\E_{\bar \Sigma} = \bar \Sigma$.
\end{example}

Recall that for every point $x \in G_\Sigma$ there exists a point $y \in M$ (may be not unique) such that $\dist (x, y) = r$ and $B_{r}(y)\cap \Sigma=\emptyset$.
Thus all points of $\Sigma \setminus \overline{B_r(M)}$ can not be energetic and thus $\overline{\Sigma \setminus \overline{B_r(M)}}$ is so-called Steiner forest id est each connected component of it is a Steiner tree with terminal points on the $\partial B_r(M)$.
 
In the plane it makes sense to define energetic rays.
\begin{definition}
We say that a ray $ (ax] $ is the \emph{energetic ray} of the set $ \Sigma $ with a vertex at the point $ x\in \Sigma $ if there exists a sequence of energetic points $ x_k \in G_\Sigma $ such that $ x_k \rightarrow x $ and $ \angle x_kxa \rightarrow 0 $.
\end{definition}

\begin{remark}
Let $\{x_k\} \subset G_\Sigma$ and let $x\in E_\Sigma$ be the limit point of $\{x_k\}$:  $x_k \rightarrow x$.
By basic property of energetic points for every point $x_k \in G_\Sigma$ there exists a point $y_k \in M$ (may be not unique) such that $\dist (x_k, y_k) = r$ and $B_{r}(y_k)\cap \Sigma=\emptyset$. Recall that it is called that $y_k$ corresponds to $x_k$.

Let $y$ be an arbitrary limit point of the set $\{y_k\}$. Then the set $\Sigma$ does not intersect $r$-neighbourhood of $y$: $B_r(y) \cap \Sigma =\emptyset$ and the point $y$ belongs to $M$ and corresponds to $x$.   
\end{remark}

Let $[sx] \subset \Sigma$ be a simple curve. Let us define $\turn (\breve {[sx]})$ as the upper limit (supremum) over all sequences of points of the curve:
  \[
 \turn(\breve{[sx]})=\sup_{n\in \N, s\preceq t^1 \prec \dots \prec t^n \prec x} \sum_{i=2}^{n} \widehat{t^i,t^{i-1}},
  \]
 where $ t ^ i $ denotes the ray of the one-sided tangent to the curve $ \breve {[st_i]} \subset \breve {[sx[} $ at point $t_i$, and $ {t_1, \dots, t_n}$ is the partition of the curve $ \breve{[sx[} $ in the order corresponding to the parameterization, for which $s$ is the beginning of the curve and $ x $ is the end. In this case, the angle 
 $\widehat{ (t^i, t^{i+1})} \in [-\pi, \pi [$ between two rays is counted from ray $ t^i $ to ray $ t^{i+1} $; positive direction is counterclockwise.
 
Let $\breve{sx}$ lay in the sufficiently small neighbourhood of $x$. Then if $B_r(y(x))\cap \breve{[sx]} = \emptyset$, it is true that
 \[
|\turn ([sx])|  < 2 \pi.
 \]
This property is the first one which is true for the plane and false in $\R^d$ with $d>2$, so this is the main difference between planar and non-planar cases. In the plane the turn is a very useful tool, see for example the proof of Theorem~\ref{horseshoeT}~\cite{cherkashin2018horseshoe}.

The second main differ between plane and other Euclidean spaces is also concerning angles: in the plane if you know the angles $\widehat{t^i,t^{i-1}}$ for $i =2, \ldots k$ then you know the angle $\widehat{t^1,t^k}$ which is not true for $\R^d$ with $d>2$.

\subsection{Derivation in the picture}

\label{diff}

During this subsection $M$ is a planar convex closed smooth curve with the radius of curvature greater than $r$; we follow paper~\cite{cherkashin2020minimizers}. 
Assume that the left and right neighborhoods of $y \in M$ are contained in $r$-neighborhoods of different energetic points $x_1$, $x_2 \in \Sigma$. We write conditions on the behavior of $\Sigma$ in the neighborhoods of $x_1$ and $x_2$ under the assumption by moving $y$ along $M$, which follow from the optimality of $\Sigma$.

We start with a preliminary lemma.

\begin{lemma}\label{lm:energetic_points_angles}
\begin{itemize}
    \item [(i)] Assume that $x$ is an isolated energetic point of degree 1 (i.e. $x$ is the end of the segment $[xz] \subset \Sigma$) with a unique corresponding point $y(x) \in M$. Then $x$, $z$ and $y(x)$  lie on the same line.
    
    \item [(ii)] Assume that $x$ is an isolated energetic point of degree 2 (i.e. $x$ is the end of distinct segments $[xz_1]$ and $[xz_2] \subset \Sigma$) with a unique corresponding $y(x)$. Then $\angle z_1xy(x) = \angle y(x)xz_2$.
\end{itemize}
\end{lemma}

Recall that $M_r$ denotes the inner part of the boundary of $B_r(M)$.
Let $S$ be the closure of a connected component of $\Sigma \setminus \conv M_r$. 
For a connected $\sigma \subset \Sigma$ define the part  of $M$ which is covered by $\sigma$ as $Q(x) = \overline{B_r(x)} \cap M$; by the restrictions on $M$ the set $Q(x)$ is always an arc of~$M$.
We denote one of the ends of the arc $Q(S) := \overline{B_r(S)} \cap M$ by $y_1$. 
Let $x \in \left(\partial B_r(y_1) \setminus \conv M_r \right) \cap S$ be the energetic point for which $y_1$ is corresponding.
Clearly $S$ is a locally minimal tree for its energetic points and points from $S\cap M_r$.
Also it is clear that $x$ cannot have more than two corresponding points. If there are two corresponding points, we denote the second by $y_2$. 
Let us denote the degree of $x$ by $d\in \{1,2\}$, and the number of points corresponding to $x$ by $k \in \{1,2\}$. Thus, there are 4 possible cases, each of which we will consider in detail below.

Let us fix a $l>0$ such that $\overline{B_l(x)} \cap \Sigma$ is the union of $d$ segments of the form $[z_ix]$, $z_i \in \partial B_l(x)$, $1 \leq i \leq d$.
For a sufficiently small $0 \leq \varepsilon < \varepsilon_0(l, r, \{y_j\}_{j=1}^k, \{z_i\}_{i=1}^d, x)$ denote by $y_1 ^\varepsilon$ the point obtained by shifting the point $y_1$ along $M$ by $\varepsilon$ (that is, such that the arc $M$ with ends at $y_1$ and $y_1^\varepsilon$ has length $\varepsilon$) in such a direction that $y_1^\varepsilon \not\in Q(S)$. For a sufficiently small in modulus $0 > \varepsilon > -\varepsilon_0(l, r, \{y_j\}_{j=1}^k \{z_i\}_{i=1}^d, x)$ we denote by $y_1^\varepsilon$ the point obtained by shifting $y_1$ along $M$ by $-\varepsilon$ in the opposite direction (that is, in such a way that $y_1^\varepsilon \in Q(S)$). Let us denote $y_1^0=y_1$. In the case of $k=2$, we denote $y_2^\varepsilon = y_2$ for any $\varepsilon$.
Put
\[
\Gamma(\varepsilon) = \min_{x'} \sum_{i=1}^d |z_ix'|,
\]
where the minimum is taken over all points $x'$ such that $|y_j^\varepsilon x'|=r$ for every $1 \leq j \leq k$. Let us denote by $x_\varepsilon$ a point at which the value $\Gamma(\varepsilon)$ is reached.

Note that $x_0 = x$, since $\Sigma$ is a minimizer. The derivative $\Gamma(\varepsilon)$ at the origin $\Gamma'(0)$ will be called \textit{the derivative of the length $\Sigma$ in the neighborhood of the point $x$ as $y_1^\varepsilon$ moves along $M$}. We calculate this derivative in each of the four cases.

The following proposition describes the possible situation to apply some calculus of variation. 

\begin{proposition} 
Let $y \in M$ be a point such that $B_r(y) \cap \Sigma = \emptyset$, and $\partial{B_r(y)}$ contains energetic points $x_1$ and $x_2$.
Define $n_y = \partial B_r(y) \cap M_r$. Then
\begin{itemize}
    \item [(i)] points $x_1$ and $x_2$ lie on opposite sides of the line $(yn_y)$;
    \item [(ii)] derivatives of the length of $\Sigma$ in neighborhoods of $x_1$ and $x_2$ when moving $y$ along $M$ are equal.
\end{itemize}
\label{diffproposition}
\end{proposition}

All the angles defined in the following cases are assumed to belong to $[0,\pi/2]$.

\paragraph{Case 1: $d = 1$, $k = 1$.}  Let the angle between $(xy_1]$ and $M$ be equal to $\alpha$. Then $\Gamma'(0) = \cos \alpha$.

\paragraph{Case 2: $d = 2$, $k = 1$.} Let the angle between $(xy_1]$ and $M$ be equal to $\alpha$ and let $\beta = \frac{1}{2} \angle z_1xz_2$.
Then $\Gamma'(0) = 2 \cos \alpha \cos \beta$.

For the next two cases we need more notation. Triangle $xy_1y_2$ is isosceles, denote $\angle xy_1y_2 = \angle xy_2y_1 =: \alpha$. Let us introduce the following coordinates: midpoint $o$ of the segment $[y_1y_2]$ is the origin of coordinates; $X$ axis is aligned with the beam
$[y_2y_1)$; the $Y$ axis is codirected with the ray $[ox)$. In particular, we have
\[
o = (0,0), \quad x = (0, r\sin\alpha), \quad y_1 = (r\cos\alpha,0), \quad y_2 = (-r\cos\alpha,0) .
\]
Let the angle between straight line $(y_1y_2)$ and $M$ at point $y_1$ be equal to $\delta$. 

\paragraph{Case 3: $d = 1$, $k = 2$.} Let us denote by $\beta$ the angle between $[z_1x]$ and the $X$ axis. Then
\[
\Gamma'(0) = \frac{\cos(\alpha + \delta)\sin(\alpha+\beta)}{\sin(2\alpha)}.
\]

\paragraph{Case 4: $d = 2$, $k = 2$.}
 As in case 3, let $\beta$ denote the angle between $[z_1x]$ and the $X$ axis; similarly, we denote by $\gamma$ the angle between $[z_2x]$ and the $X$ axis. Then
\[
\Gamma'(0) = \frac{\cos(\alpha + \delta)}{\sin(2\alpha)} (\sin(\alpha+\beta) + \sin(\alpha+\gamma)).
\]

If $M$ is piecewise smooth one can also apply such a type of derivation, in particular it is heavily used in the proof of Theorem~\ref{rectangleT}.

\subsection{Convexity argument}

Suppose that we fix some $M_0 \subset M \subset \mathbb{R}^d$ and consider a (possibly infinite) abstract tree $T$ whose vertices are encoded by points of $M_0$.
Let us pick an arbitrary point from $\overline{B_r(m)}$ for every $m \in M_0$ and connect such points by segments with respect to $T$.
Consider the length $L$ of such a representation of $T$; note that we allow the representation to contain cycles or edges of zero length.

Then $L$ is a convex function from $(\mathbb{R}^d)^{M_0}$ to $\mathbb{R}$. 
Also if $v,u \in \overline{B_r(m)}$, then $\alpha v + (1-\alpha)u$ also lies in $\overline{B_r(m)}$. 
It implies that the sets of local and global minimums of $L$ coincide and form a convex set. It usually means that $L$ is a unique local minimum.

This approach allows us to show that if one fixes a topology of a solution, then the corresponding Steiner-type problem has a unique solution. The proofs of Theorems~\ref{theocornerpoints} and~\ref{theo:finiteinverse} heavily use it.

\subsection{Lower bounds on the length of a minimizer}
The proof of the following folklore inequality can be found, for instance in~\cite{mosconi2005gamma}. 

\begin{lemma}
Let $\gamma$ be a compact connected subset of $\mathbb{R}^d$ with $\H(\gamma) < \infty$.
Then
\[
\HH^d(\{x \in \mathbb{R}^d: \dist(x,\gamma) \leq t\} ) \leq \H (\gamma) \omega_{d-1}t^{d-1} + \omega_d t^d,
\]
where $\omega_k$ denotes the volume of the unit ball in $\mathbb{R}^k$.
\end{lemma}

The following corollary is very close to a theorem of Tilli on average distance minimizers~\cite{tilli2010some}.

\begin{corollary}
\label{TilliType}
Let $V$ and $r$ be positive numbers. Then for every set $M$ with $\HH^d(M) = V$ a maximal distance $r$-minimizer has the length at least
\[
\max \left (0, \frac{V - \omega_d r^d}{\omega_{d-1}r^{d-1}} \right).
\]
\end{corollary}

Theorem~\ref{Th:C11isaminimizer} follows from the fact that for a $C^{1,1}$-curve and small enough $r$ the inequality in Corollary~\ref{TilliType} is sharp.
Let us provide a lower bound from~\cite{inverse2022} on the length of a minimizer in the planar case.

\begin{proposition}
Let $M$ be a planar convex set and $\Sigma$ is an $r$-minimizer for $M$. Then 
\[
\H(\Sigma) \geq \frac{\H(\partial M) - 2\pi r}{2}.
\]
\label{Cor:perimeter}
\end{proposition}

\subsection{Results on Steiner trees}

Recall that the Steiner tree problem is defined in Subsection~\ref{sec:Sttree}.

We call a topology $T$ \textit{realizable} for a set $P \subset \mathbb{R}^d$ if there exists such a locally minimal tree $S(P)$ with topology $T$; we will denote this tree by $S_T(P)$; by the following proposition the notation is correct.

\begin{proposition}[Melzak,~\cite{melzak1961problem}]
If a topology $T$ is realizable for $P$ then the realization $S_T(P)$ is unique.
\label{melzakuniq}
\end{proposition}

It is well-known that the length of a Steiner tree depends only on the directions at the terminals (see Maxwell formula in~\cite{gilbert1968steiner}).
So the following theorem is a crucial part of the proof of Proposition~\ref{prop:notaminimizer}.

\begin{theorem}[Oblakov~\cite{oblakov2009non}]
There are no two distinct topologies $T_1$ and $T_2$ and a planar configuration $P$ such that locally minimal trees $S_{T_1}(P)$ and $S_{T_2}(P)$ are codirected at terminals.
\label{oblakov}
\end{theorem}

\section{New results}

\subsection{\texorpdfstring{$\Gamma$}{Gamma}-convergence}
$\Gamma$-convergence is an important tool in studying minimizers based on approximation of energy. For Euclidean space the following definition of $\Gamma$-convergence can be used.
Let $X$ be a first-countable space (every Euclidean space is first-countable) and $F_{n}$: $X\to \overline {\mathbb {R} }$ be a sequence of functionals on $X$. Then $F_{n}$ are said to \textit{$\Gamma$-converge} to a $\Gamma$-limit $F$ : $X\to \overline {\mathbb {R} }$ if the following two conditions hold.

\begin{itemize}
    \item Lower bound inequality. For every sequence $x_{n}\in X$ such that $x_n\to x$ as $ n\to +\infty $,
    \[  
    F(x)\leq \liminf _{{n\to \infty }}F_{n}(x_{n}).
    \]
    \item Upper bound inequality. For every $x\in X$, there is a sequence $x_{n}$ converging to $x$ such that
    \[
    F(x)\geq \limsup _{{n\to \infty }}F_{n}(x_{n}).
    \]
\end{itemize}

In the case of maximal distance minimizers for a given compact set $M$ and a number $l>0$ we can consider the space $X$ of connected compact sets with one-dimensional Hausdorff measure at most $l$, equipped with the Hausdorff distance (the distance $d_H$ between $A, C \in X$ is the smallest $\rho$ such that $A \subset \overline{B_\rho(C)}$ and $C \subset \overline{B_\rho(A)}$).

\begin{proposition} \label{pr:convergence}
If a sequence of compacts $M_i$ converges to $M$ then $F_{M_i}$ $\Gamma$-converges to $F_M$.
\end{proposition}

\begin{proof}
By the definition of $F_{M}$ and the triangle inequality we have    
\begin{equation} \label{eq:convergence}
    |F_{M_i} (S_i) - F_{M} (S)| \leq  |F_{M_i} (S_i) - F_{M} (S_i)| + |F_{M} (S_i) - F_{M} (S)| \leq d_H(M_i,M) + d_H(S_i,S)
\end{equation}
for every connected $S_i$ and $S$. So by~\eqref{eq:convergence} for every sequence of $S_i$ with limit $S$ we have the first condition of $\Gamma$-convergence holds.
For the second condition consider $S_i$ being a Steiner tree for a finite $1/i$-network $N_i \subset S$.
By the definition $\HH(S_i) \leq \HH(S) \leq l$. Again, by~\eqref{eq:convergence} $F_{M_i} (S_i)$ converges to $F_M(S)$.
 
\end{proof}

\subsection{Approximation by Steiner trees}

A crucial property of $\Gamma$-convergence is that in the notation of Proposition~\ref{pr:convergence} every limit point of the sequence of minimizers of $F_{M_i}$ is a minimizer of $F_M$.
Now let $M_n$ be a finite $1/n$-network for $M$, so that every minimizer for $M_n$ is a finite Steiner tree.

Unfortunately, in the case of several minimizers for $M$ we cannot be sure that every minimizer is approximated.
On the other hand it can be approximated a posteriori. 
Let $\Sigma$ be a minimizer for $M$ and let $\mathcal{E}_k\subset \Sigma$ be a finite $1/k$-network and $\Sigma_k$ be an arbitrary solution to the Steiner problem for $\mathcal{E}_k$. By the definition we have
    \[
        \H(\Sigma_k)\leq \H(\Sigma).
    \] 
    On the other hand, for any subsequential limit (with respect to the Hausdorff distance) $\Sigma'$ of the sequence $\Sigma_k$ we have $\Sigma \subset \Sigma'$ and so
    \[ 
        \H(\Sigma)\leq \H(\Sigma')\leq \liminf_{k\to \infty} \H(\Sigma_k)
    \]
by Go{\l}{\k{a}}b's theorem. It follows that $\Sigma_k$ converges to $\Sigma$ and $\H(\Sigma_k)$ converges to $\H(\Sigma)$.

Summing up, every maximal distance minimizer is a limit of finite Steiner trees.
Similar results are also proved in~\cite{alvarado2020maximum}.
Recall that more detailed and structural relations of finite Steiner trees and maximal distance minimizers are considered in Section~\ref{sec:Sttree}.

\subsection{NP-hardness} \label{sec:NP}

It is well-known that Euclidean Steiner problem is NP-hard~\cite{garey1977complexity} even if we restrict the terminals to two lines in the plane~\cite{rubinstein1997steiner}. 
The first source of hardness is that if we fix a topology in planar Steiner tree problem, then one can write the length in the explicit form.
However the expression for $n$ terminals may have $\Omega (n)$ square roots.

To avoid it Garey, Graham and Johnson~\cite{garey1977complexity} introduce a discrete version of the Steiner problem: all terminals and branching points are forced to have integer coordinates and the length of every segment is replaced with its ceiling. 
Of course a minimizer of a new problem does not inherit any geometric properties, in particular we have no $2\pi/3$-condition at a branching point.
Such a discretization appears to be NP-complete (and so the initial one is NP-hard), namely, Garey, Graham and Johnson used a reduction of the X3C problem to this version of the Steiner problem.  
The X3C problem is to decide whether a family of 3-sets $\mathcal{F} \subset 2^{[3n]}$
has a subfamily of $n$ sets which covers $[3n]$ (as usual, $[3n] = \{1,2,\dots,3n\}$). It is well-known that X3C is NP-complete.

First we need the following reduction to the classical Steiner problem.

\begin{theorem}[Garey--Graham--Johnson~\cite{garey1977complexity}] \label{th:GGJ}
For a given $\mathcal{F} \subset 2^{[3n]}$ one can construct in a polynomial time in $n$ an input $X(\mathcal{F}) \subset \mathbb{R}^2$ whose size is also polynomial in $n$ such that
\begin{itemize}
    \item [(i)] if $\mathcal{F}$ has an $n$-set covering then a solution to the Steiner problem for $X(\mathcal{F})$ has the length at most $L$;
    \item [(ii)] if $\mathcal{F}$ does not have an $n$-set covering then a solution to the Steiner problem for $X(\mathcal{F})$ has the length at least $L + 12|X(\mathcal{F})|$.
\end{itemize}
Moreover $L = L(\mathcal{F})$ can be extracted from the construction of $X(\mathcal{F})$ in an explicit form.
\end{theorem}

Now let us repeat Garey--Graham--Johnson rounding in the case of maximal distance minimizers. The following problem is a discrete approximation of Problem~\ref{TheDualProblem} analogous to the discrete version of Steiner problem used in~\cite{garey1977complexity}.
Following~\cite{garey1977complexity} we replace the length function with its ceiling because it is not known if the problem of determining whether 
$\sum \sqrt{n_i} < L$ is NP or not ($n_i$, $L$ are given integers).

\begin{problem} \label{DMDM}
    Let $M$ be a finite set of points in the plane with integer coordinates and $r, \ell \in \mathbb{N}$.
    Decide whether exists a connected graph whose vertices have integer coordinates and edges are segments with the sum of the ceiling function of the length over edges at most $\ell$ such that every point of $M$ lies at a distance at most $r$ from some vertex of the graph.
\end{problem}

Now we are ready to obtain the following corollary of Garey--Graham--Johnson results and the approximation.

\begin{proposition}
Problem~\ref{DMDM} is NP-complete.     
\end{proposition}

\begin{proof}
Let $\mathcal{F} \subset 2^{[3n]}$ be an arbitrary family.
Consider the set $X(\mathcal{F})$ from Theorem~\ref{th:GGJ}. 
Fix any $r\in \mathbb{N}$ and let $k > 10r|X(\mathcal{F})|$ be a large integer number. Define $kX(\mathcal{F})$ as a set homothetic to $X(\mathcal{F})$ with the scale factor $k$. Let $M$ be the set of points closest to $kX(\mathcal{F})$ in the integer grid $\mathbb{Z}^2$.
Put also $\ell = kL(\mathcal{F}) + k$.

Then if $F$ has an $n$-set covering, then a solution $\St$ of the Steiner problem for $kX(\mathcal{F})$ has the length at most $kL(\mathcal{F})$. 
Now we replace in $\St$ every vertex with the closest point from $\mathbb{Z}^2$; denote the resulting set by $\St^D$.
By the definition $\St^D$ is a graph whose vertices have integer coordinates and it connects $M$; also it has at most $2|X(\mathcal{F})| - 3$ segments.
After the rounding the length of every segment of $\St$ grows by at most $\sqrt{2}$. 
Hence the ceiling function of the length of an edge in $\St^D$ is at most the length of the corresponding edge of $\St$ plus 3.
Thus sum of the ceiling function of the length over edges $\St^D$
is at most 
\[
kL + 6|X(\mathcal{F})| < \ell
\]
and the answer to Problem~\ref{DMDM} is positive.

On the other hand let us show that in the case when $F$ has no $n$-set covering, 
the answer to Problem~\ref{DMDM} is negative.
Consider a solution $\Sigma^D$ of Problem~\ref{DMDM} for $X(\mathcal{F})$. Assume the contrary, so that the sum of the ceiling function of length over the edges of $\Sigma^D$ is at most $\ell$. It implies $\H(\Sigma^D) \leq \ell$. 
Consider the homothety $\Sigma^D_{1/k}$ of $\Sigma^D$ with the scale factor $1/k$, one has
\[
\H(\Sigma^D_{1/k}) = \frac{\H(\Sigma^D)}{k} \leq \frac{\ell}{k}.
\]
By definition for every $x \in M$ there is a point $\sigma \in \Sigma^D$ at a distance at most $r$ from $x$. 
Hence for every $x \in X(\mathcal{F})$ there is a point $\sigma \in \Sigma^D_{1/k}$ at a distance at most $(r+1)/k < 1$ from $x$. 
Thus the length of a Steiner tree for $X(\mathcal{F})$ is at most
\[
\frac{\H(\Sigma^D)}{k} + |X(\mathcal{F})| \cdot \frac{r+1}{k} \leq \frac{\ell + (r+1)\cdot |X(\mathcal{F})|}{k} \leq L + |X(\mathcal{F})|.
\]
We got a contradiction with Theorem~\ref{th:GGJ} and thus finished the reduction of the X3C problem to Problem~\ref{DMDM} with the input $M, r, \ell$.

Finally one can easily compute the sum of the ceiling function of length over edges of a competitor for Problem~\ref{DMDM} in polynomial time.
    
\end{proof}

\subsection{Penalized form}
Let $M$ be a given compact set. Let us consider a problem of minimization $F_M(S)+\lambda \H(S)$ for some $\lambda >0$, where $F_M(S)=\max_{y \in M} \dist (y, S)$ among all connected compact sets $S$. We will call this problem \textit{$\lambda$-penalized}.

Clearly every set $T$ which minimizes $\lambda$-penalized problem for some $\lambda$ is a maximal distance minimizer for a given input $M$ and the restriction of energy $r\defeq F_M(T)$. Hence the solutions to this problem inherit all regularity properties of maximal distance minimizers.

As usual in variational calculus on a restricted class, it may happen for a small variation $\Phi_\varepsilon(\Sigma)$ of $\Sigma$, that the length constraint $\H(\Phi_\varepsilon(\Sigma))\leq l$ is violated. Hence to compute Euler--Lagrange equation associated to the maximal distance minimization problem a possible way is to consider first the penalized functional $F_M(S)+\lambda \H(S)$ for some constant $\lambda$, for which any competitor $\Sigma$ is admissible without length constraint.

Hence it is also make sense to consider \textit{local penalized problem}: the problem of searching a connected compact set $S$ of a finite length, such that  $\H(S)+\lambda F_M(S) \leq \H(T)+\lambda F_M(T)$ for every connected compact $T$ with $\diam (S \triangle T)< \varepsilon$ for sufficiently small $\varepsilon>0$. The solutions to these problems also inherit properties of local maximal distance minimizers. 

\begin{proposition}
    Consider 
\[
\min_{\Sigma \text{ compact and connected}} F_M(\Sigma)+\lambda(\H(\Sigma)-l)^{+}
\]
for any constant $\lambda>1$. Then this problem is equivalent to the maximal distance minimization problem.
\end{proposition}
\begin{proof}
    The same as for average distance minimizers (see Proposition~23 in~\cite{lemenant2011presentationENG}). 
    The only difference is that we use an obvious inequality $F_M(S \setminus T_\varepsilon) - F_M(S) \leq \varepsilon$ for connected sets $S$ and $S \setminus T_\varepsilon$ such that $\H(T_\varepsilon) = \varepsilon$.
\end{proof}

\subsection{Uniqueness}

Let us start with the following simple observation.
The set of minimizers for $M$ being a circle $\partial B_R(o)$ is uncountable for $r < R$.
Indeed, any minimizer has no loops and does not reduce to a point, so its rotations rarely coincide.

Note that for every compact $M \subset \mathbb{R}^d$ and $r$ equal to the radius of the smallest ball containing $M$,
there is a unique point $o$ such that $M \subset \overline{B_r(o)}$, id est the solution to Problem~\ref{TheDualProblem} is unique.
For a larger $r$ one has an uncountable number of solutions.
This motivate us to consider only small enough $r$.
Let us call a finite point configuration $M$ \textit{ambiguous} if Problem~\ref{TheDualProblem} has several
solutions for $M$ and $r < r_0(M)$.
The following statement is a straightforward corollary from the main theorem of~\cite{basok2018uniqueness}.

\begin{proposition} \label{pr:uniq}
For $n \geq 4$ the set of planar $n$-point ambiguous configurations $M$ has Hausdorff dimension $2n-1$ (as a subset of $\mathbb{R}^{2n}$).
\end{proposition}

\begin{proof}
Fix $n \geq 4$. The main result of~\cite{basok2018uniqueness} states that the Hausdorff dimension of planar $n$-point configurations with multiple Steiner trees is $2n-1$. 

Recall that a topology $T$ of a labelled Steiner tree is the corresponding abstract graph; a topology is \textit{full} if every terminal has degree 1. We call a topology \textit{generic} if it has no terminals of degree 3.
For a not full generic topology $R$ one can replace vertex $A$ of degree 2 with a Steiner point $b$ and add edge $bA$; the resulting topology $T(R)$ is full.

By Proposition~\ref{melzakuniq} the set of all configurations for which there is a realization of any degenerate topology has Hausdorff dimension $2n-2$.

Let us show that if a Steiner tree for a finite $M$ is unique then $M$ is not ambiguous.
Consider any $n$-point planar configuration $M$ with a unique Steiner tree $\St$ whose topology is generic.
Let the length of the second locally minimal tree be $\H(\St) + a$ and choose $r < a/(2n)$.

Then a maximal distance minimizer for a given $M$ and $r$ is obtained by a convexity argument for a topology $T$.
Thus the Hausdorff dimension of planar $n$-point ambiguous configuration is at most $2n-1$.

To show that the Hausdorff dimension $2n-1$ of the set of planar $n$-point ambiguous configurations is at least $M$
we word-by-word repeat the argument of Lemma 8 in~\cite{basok2018uniqueness}.

\end{proof}

Note that we need $n \geq 4$ in the proof since there is only one full topology for each $n \leq 3$.

\section{Open questions}

\subsection{Regularity}

The first question, especially if the answer is negative, might be difficult. 

\begin{question}
Does there exist a nonplanar maximal distance minimizer with an infinite number of branching points?
\end{question}

The next question concerns the Ahlfors regularity in the following strong form for $d > 2$.

\begin{question}
Are there exist constants $C_1,C_2$ depending only on $d$ such that for every $M$ and $r > 0$ there is a $\varepsilon_0 > 0$ for which
\[
C_1\varepsilon \leq \H(\Sigma \cap B_\varepsilon(x))\leq C_2\varepsilon
\]
holds for every $x \in \Sigma$ and $\varepsilon < \varepsilon_0$?
\end{question}

An easier question should be to construct an example of a minimizer with a branching point, whose neighbourhood does not coincide with a regular tripod:

\begin{question}
To construct a (nonplanar) maximal distance minimizer $\Sigma$ containing a locally nonplanar branching point $x$,  
i.e. for every $\varepsilon > 0$ the set $B_\varepsilon(x) \cap \Sigma$ does not belong to a plane.
\end{question}

Thus the question if there exists a nonplanar maximal distance minimizer with an infinite number of points with three tangent rays also makes sense. 

The following question asks if one-sided tangents should have continuity from the corresponding side.
 
\begin{question}
Does Lemma~\ref{zapred2} hold for a $d$-dimensional maximal distance minimizer?
\end{question}

All the questions in this subsection can be also asked for local minimizers.

\subsection{Explicit solutions}

Recall that the horseshoe conjecture is still open.

\begin{question}
Find maximal distance minimizers for a circumference of radius $4.98r>R > r$.
\end{question}

At the same time, the statement of Theorem~\ref{horseshoeT} does not hold for a general $M$ if the assumption on the minimal radius of curvature is omitted as we show below.
 
Define a \textit{stadium} to be the boundary of the $R$-neighborhood of a segment. By the definition, a stadium has the minimal radius of curvature $R$. 
Let us show that if $R < 1.75r$ and a stadium is long enough, then there is the connected set $\Sigma'$ that covers $M$ and has the length smaller than an arbitrary horseshoe. 
 
\begin{figure}[h]
    \centering

\begin{tikzpicture}

\foreach \y in {0, 3.51389977019889} {
    \draw[very thick, loosely dotted] 
        (-1, \y) -- (0, \y);
    \draw[very thick]
        (0, \y) -- (3.66, \y);
    \draw[very thick, loosely dotted] 
        (3.66, \y) -- (4.66, \y);
}
\draw[very thick] (-1, 0) arc (270:90:1.757);
\draw[very thick] (4.66, 0) arc (-90:90:1.757);

\draw[thin, dashed] (1.8298959646083486 + 1, 0.0) arc (0:180:1);

\draw[blue, thick] 
(0.0, 0.9652268503449002) -- 
(1.5684822553785844, 0.9652268503449002) --
(1.8298959646083486, 1.11615412573856) --
(2.0913096738381127, 0.9652268503449002) -- 
(3.659791929216697, 0.9652268503449002);

\node[blue] at (1.61,1.8) {$\Sigma_0$};

\draw[dashed] (0.0, 0.0) -- (0.0, 3.51389977019889);
\draw[dashed] (3.659791929216697, 0.0) -- (3.659791929216697, 3.51389977019889);

\fill[blue] (0.0, 0.9652268503449002) 
    circle (1.5pt) node[black, below right] {$a$};
\fill[blue] (3.659791929216697, 0.9652268503449002) 
    circle (1.5pt) node[black, below left] {$b$};

\draw[blue, thick] 
(1.8298959646083486, 1.11615412573856) -- 
(1.8298959646083486, 2.58208251632886);

\draw[thin, dashed] (1, 3.51389977019889) arc (0:-90:1);
\draw[thin, dashed] (2.66, 3.51389977019889) arc (180:270:1);

\draw[blue, thick]
(1.8298959646083486, 2.58208251632886) -- 
(0.9149479823041743, 3.11032798020668);

\draw[blue, thick] 
(1.8298959646083486, 2.58208251632886) -- 
(2.744843946912523, 3.11032798020668);

\end{tikzpicture}
   \caption{Horseshoe is not a minimizer for long enough stadium with $R < 1.75r$.}
    \label{stadion}
\end{figure}
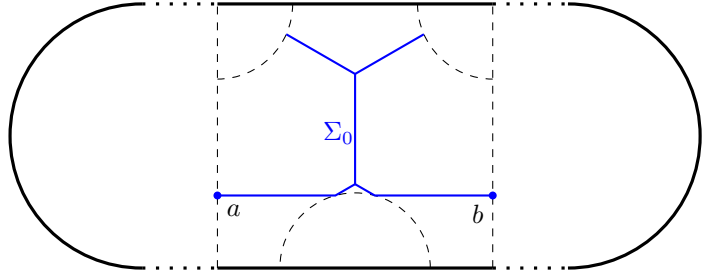

Define $\Sigma_0$ to be the locally minimal tree depicted in Fig.~\ref{stadion}. 
Let $\Sigma'$ consist of copies of $\Sigma_0$, glued at points $a$ and $b$ along the stadium. Note that $F_M(\Sigma') \leq r$ by the construction.
In the case $R < 1.75r$ the length of $\Sigma_0$ is strictly smaller than $2|ab|$. 
Thus for a long enough stadium $\Sigma'$ has length $\alpha L + O(1)$, where $L$ is the length of the stadium and $\alpha < 2$ is a constant depending on $\Sigma_0$ and $R$. 
On the other hand, any horseshoe has length $2L + O(1)$. 

This example leads to the following problems.

\begin{question}
Find the minimal $\alpha $ such that Theorem~\ref{horseshoeT} holds with the replacement of $5r$ with $\alpha r$.
\end{question}
 
\begin{question}
Describe the set of $r$-minimizers for a given stadium.
\label{problemstadium}
\end{question}

Analogously to the stadium case one can easily show that for a narrow rectangle (i.e. with a sufficiently small $\frac{|a_1a_2|}{|a_2a_3|}$) and some $r>0$ a minimizer should have another topology than depicted at~Fig.~\ref{rectangle}.

Also one may consider the following relaxation of Problem~\ref{problemstadium}.
\begin{question}
Fix a real $a > 2r$. Let $M(l)$ be the union of two sides of length $l$ of a rectangle $a \times l$ and $\Sigma(l)$ be a minimizer for $M(l)$. Find
\[
s(a) := \lim_{l \to \infty} \frac{\H (\Sigma(l))}{l}.
\]
\end{question}

If $a > 10r$ one may add up $M(l)$ to a stadium and use Theorem~\ref{horseshoeT} to get $s(a) = 2$.

\subsection{Uniqueness}

Recall that if $\Sigma$ be an $r$-minimizer for some $M$, then it is a minimizer for $\overline{B_r(\Sigma)}$.
This motivates the following question.

\begin{question}
Let $\Sigma$ be an $r$-minimizer for some $M$. Is $\Sigma$ the unique $r$-minimizer for $\overline{B_r(\Sigma)}$?
\end{question}
A weaker form of this question is if we replace $r$ with some positive $r_0 < r$ in the hypothesis.

Again we are interested whether Proposition~\ref{pr:uniq} holds in larger dimensions.

\begin{question}
Fix $d \geq 3$ and $n \geq 4$.
Find the Hausdorff dimension of $d$-dimensional $n$-point ambiguous configurations $M$ (as a subset of $\mathbb{R}^{dn}$).
\end{question}

A weaker question is to determine whether the set of $d$-dimensional $n$-point ambiguous configurations has measure zero.

Recall that we believe that every Steiner tree which is not unique for its $n$ terminals cannot be a maximal distance minimizer for any $n$-point set $M$ and any $r > 0$.
This motivated the following more general question.

\begin{question}
Fix $M \subset \mathbb{R}^d$ and $r > 0$.
Does a set $G_\Sigma$ of energetic points determine a minimizer $\Sigma$?
\end{question}

\paragraph{Ackowledgements.} We thank Alexey Gordeev for the discussions.
Yana Teplitskaya is supported by the Simons Foundation grant 601941, GD.

\bibliographystyle{plain}
\bibliography{main.bib}

\end{document}